\documentclass[12pt, reqno]{amsart}

\usepackage{amssymb,amsmath,amsthm}
\usepackage{graphicx}
\usepackage{url}
\usepackage{xparse}

\makeatletter
\let\@@citation@@=\citation
\renewcommand{\citation}[1]{\@@citation@@{#1}%
\@for\@tempa:=#1\do{\@ifundefined{cit@\@tempa}%
  {\global\@namedef{cit@\@tempa}{}}{}}%
}
\makeatother

\usepackage{hyperref}

\makeatletter
\def\@lbibitem[#1]#2#3\par{%
  \@ifundefined{cit@#2}{}{\@skiphyperreftrue
  \H@item[%
    \ifx\Hy@raisedlink\@empty
      \hyper@anchorstart{cite.#2\@extra@b@citeb}%
        \@BIBLABEL{#1}%
      \hyper@anchorend
    \else
      \Hy@raisedlink{%
        \hyper@anchorstart{cite.#2\@extra@b@citeb}\hyper@anchorend
      }%
      \@BIBLABEL{#1}%
    \fi
    \hfill
  ]%
  \@skiphyperreffalse}%
  \if@filesw
    \begingroup
      \let\protect\noexpand
      \immediate\write\@auxout{%
        \string\bibcite{#2}{#1}%
      }%
    \endgroup
  \fi
  \ignorespaces
  \@ifundefined{cit@#2}{}{#3}}

\def\@bibitem#1#2\par{%
  \@ifundefined{cit@#1}{}{\@skiphyperreftrue\H@item\@skiphyperreffalse
  \Hy@raisedlink{%
    \hyper@anchorstart{cite.#1\@extra@b@citeb}\relax\hyper@anchorend
    }}%
  \if@filesw
    \begingroup
      \let\protect\noexpand
      \immediate\write\@auxout{%
        \string\bibcite{#1}{\the\value{\@listctr}}%
      }%
    \endgroup
  \fi
  \ignorespaces
  \@ifundefined{cit@#1}{}{#2}}
\makeatother

\newcommand\reals{{\mathbb R}}

\newcommand\cH{{\mathcal H}}

\newcommand\lex{{\mathrm{ex}_{<}'}}
\newcommand\ex{{\mathrm{ex}}}

\newcommand\scalemath[2]{\scalebox{#1}{\mbox{\ensuremath{\displaystyle #2}}}}
\DeclareDocumentCommand{\path}{o o o o o o o o o}  
{%
  \scalemath{0.5}{\bullet\!\!\overset{\displaystyle #1}{-}\!\!\bullet
  \IfValueT{#2}{\!\!\overset{\displaystyle #2}{-}\!\!\bullet}
  \IfValueT{#3}{\!\!\overset{\displaystyle #3}{-}\!\!\bullet}
  \IfValueT{#4}{\!\!\overset{\displaystyle #4}{-}\!\!\bullet}
  \IfValueT{#5}{\!\!\overset{\displaystyle #5}{-}\!\!\bullet}
  \IfValueT{#6}{\!\!\overset{\displaystyle #6}{-}\!\!\bullet}
  \IfValueT{#7}{\!\!\overset{\displaystyle #7}{-}\!\!\bullet}
  \IfValueT{#8}{\!\!\overset{\displaystyle #8}{-}\!\!\bullet}
  \IfValueT{#9}{\!\!\overset{\displaystyle #9}{-}\!\!\bullet}}
}

\def\extractfirst#1#2\relax{#1}
\def\extractsecond#1#2#3\relax{#2}
\def\extractthird#1#2#3#4\relax{#3}
\def\extractfourth#1#2#3#4#5\relax{#4}
\def\extractfifth#1#2#3#4#5#6\relax{#5}
\def\extractsixth#1#2#3#4#5#6#7\relax{#6}
\def\extractseventh#1#2#3#4#5#6#7#8\relax{#7}
\def\extracteight#1#2#3#4#5#6#7#8#9\relax{#8}

\newcommand\ut[1]{\path[{\expandafter\extractfirst#1\relax}][{\expandafter\extractsecond#1\relax}]}

\theoremstyle{plain}
\newtheorem{theorem}{Theorem}[section]
\newtheorem{lemma}[theorem]{Lemma}
\newtheorem{corollary}[theorem]{Corollary}

\newtheorem{proposition}[theorem]{Proposition}

\theoremstyle{definition}

\newtheorem{defn}[theorem]{Definition}

\newtheorem{claim}{Claim}

\newtheorem*{remark}{Remark}

\usepackage{enumitem}

\newlist{Case}{enumerate}{1}
\setlist[Case]{label=Case \arabic*:}

\newcommand\cref[1]{Corollary~\ref{cor:#1}}

\def\sgn{\mathop{\rm sign}\nolimits}

\title{Tur\'an problems for Edge-ordered graphs}

\author{D\'aniel Gerbner}
\address{D\'aniel Gerbner, Alfr\'ed R\'enyi Institute of Mathematics, Budapest} 
\email{gerbner@renyi.hu}

\author{Abhishek Methuku}
 \address{Abhishek Methuku,  School of Mathematics, University of Birmingham, Birmingham}
\email{abhishekmethuku@gmail.com}

\author{D\'aniel T. Nagy}
\address{D\'aniel T. Nagy,  Alfr\'ed R\'enyi Institute of Mathematics, Budapest} 
\email{nagydani@renyi.hu}
 
\author{D\"om\"ot\"or P\'alv\"olgyi}
\address{D\"om\"ot\"or P\'alv\"olgyi,  MTA-ELTE Lend\"ulet Combinatorial Geometry Research Group, Institute of Mathematics, E\"otv\"os Lor\'and University (ELTE), Budapest}
\email{dom@cs.elte.hu}

\author{G\'abor
  Tardos}
  \address{G\'abor
  	Tardos, Alfr\'ed R\'enyi Institute of Mathematics, Budapest  \and Central European University, Budapest, Vienna \and  Moscow Institute of Physics and Technology, Dolgoprudny}
\email{tardos@renyi.hu}

 \author{M\'at\'e Vizer}
 \address{M\'at\'e Vizer,  Alfr\'ed R\'enyi Institute of Mathematics, Budapest} 
 \email{vizermate@gmail.com}

\date{\today}

\usepackage{fancyhdr}
\begin{document}

\begin{abstract}
In this paper we initiate a systematic study of the Tur\'an problem for edge-ordered graphs. A simple graph is called \emph{edge-ordered}, if its edges are linearly ordered. An isomorphism between edge-ordered graphs must respect the edge-order. A subgraph of an edge-ordered graph is itself an edge-ordered graph with the induced edge-order. We  say that an edge-ordered graph $G$ \emph{avoids} another edge-ordered graph $H$, if no subgraph of $G$ is isomorphic to $H$. 

\vspace{1mm}
The \emph{Tur\'an number} of an edge-ordered graph $H$ is the maximum number of edges in an edge-ordered graph on $n$ vertices that avoids $H$. We study this problem in general, and establish an Erd\H os-Stone-Simonovits-type theorem for edge-ordered graphs -- we discover that the relevant parameter for the Tur\'an number of an edge-ordered graph is its \emph{order chromatic number}. We establish several important properties of this parameter. 

\vspace{1mm}
We also study Tur\'an numbers of edge-ordered paths, star forests and the cycle of length four. We make strong connections to Davenport-Schinzel theory, the theory of forbidden submatrices, and show an application in Discrete Geometry.
\end{abstract}
\maketitle

\section{Introduction}

The most basic \textit{Tur\'an-type extremal problem} asks the
maximum number $\ex(n,H)$ of edges in an $n$ vertex simple graph that does not
contain a ``forbidden'' graph $H$ as a subgraph. For a family $\cH$ of forbidden graphs we write $\ex(n,\cH)$ to denote the maximal number of edges of a simple graph on $n$ vertices that contains no member of $\cH$ as a subgraph. This problem has its roots in the
works of Mantel, \cite{M1907} and Tur\'an, \cite{T1941}, where they considered the
case where the forbidden graph is a complete graph. For a survey see
F\"uredi and Simonovits, \cite{FS2013}. Several extensions of Tur\'an-type
extremal problems for graphs have been
studied. For a survey on extremal hypergraph problems see Keevash, \cite{K2011}. The extremal theory of graphs with a circular or
linear order on their \emph{vertex set} has a rich history. For example, see Bra\ss, K\'arolyi, Valtr,
\cite{BKV2003} or Tardos, \cite{T2018}, respectively. In this paper we initiate a systematic study of Tur\'an-type problems for \emph{edge-ordered graphs} and establish several fundamental results.

\vspace{1mm}

An \emph{edge-ordered graph} is a finite simple graph $G = (V, E)$ with a linear
order on its edge set $E$. We often give this linear order with a
\emph{labeling} $L:E\to\reals$ (that we also call \emph{edge-ordering} or \emph{edge-order}, in short). In this case we denote by $G^L$ the edge-ordered graph
obtained, and we also call it a \emph{labeling} of $G$. Note that we always assume the function $L:E\to\reals$ is injective (so that it defines a linear order on the edges) and
we use the labeling only to define this edge-order, so $G^L$ and $G^{L'}$
represent the same edge-ordered graph if any pair of edges $e,f\in E$, $L(e)<L(f)$ holds if and only if $L'(e)<L'(f)$ holds. 

\vspace{1mm}

An isomorphism between edge-ordered graphs must respect the edge-order. A subgraph of an edge-ordered graph is itself an edge-ordered graph with the induced edge-order.
We say that the edge-ordered graph $G$ \emph{contains} another edge-ordered
graph $H$, if $H$ is isomorphic to a subgraph of $G$. Otherwise we say that $G$
\emph{avoids} $H$. We say that $G$ \emph{avoids} a family of edge-ordered
graphs, if it avoids every member of the family. When speaking of a family of
edge-ordered graphs we always assume that all members of the family are \emph{non-empty}, that is they have at least one edge. This is necessary for the definition of the Tur\'an number below to make sense. Note that similar extremal problems for vertex-ordered graphs (where the linear order is on the vertices instead of edges) has been studied before, see for example \cite{Korandi, MT04, MethukuTomon, T2018}.

\vspace{2mm}

The Tur\'an problem for edge-ordered graphs can be formulated as follows.

\begin{defn} For a positive integer $n$ and a family of edge-ordered graphs $\cH$, let the
  \emph{Tur\'an number of $\cH$} be the maximal number of edges in an
  edge-ordered graph on $n$ vertices that avoids $\cH$, and let this maximum be denoted by $\lex(n,\cH)$. If there is only one forbidden edge-ordered graph $H$, we simply write $\lex(n,H)$ instead of $\lex(n,\{H\})$.
\end{defn}

Any Tur\'an-type problem in extremal graph theory can be formulated in this language. Indeed, let $\cH$ be a family of forbidden simple graphs and define $\cH'=\{H^L:H\in\cH, \hbox{$L$ is a labeling of $H$}\}$. We clearly have
$$\ex(n,\cH)=\lex(n,\cH').$$
As a consequence, we have the following simple but useful bound for any simple graph $H$ and any labeling $L$:
$$\lex(n,H^L)\ge\ex(n,H).$$

\textbf{Notation.} We will denote the edge order of short paths and cycles by simply giving labels to the edges along the path or cycle. For example, the edge-ordering of a path $P_4$ on four vertices, say $abcd$, that gives the edge $ab$ the label 1, the edge $bc$ the label 3, and the edge $cd$ the label 2 is denoted by $P_4^{132}$. (In other words, this labeling denotes the edge-ordering $ab < cd < bc$.) Similarly, $C_4^{1234}$ denotes the cyclically increasing labeling of the cycle $C_4$.

\subsection{History} 
Only a few special instances of the Tur\'an problem for edge-ordered graphs have been investigated so far. In most of these cases the aim was to find an increasing path or trail, defined as follows: We call a sequence $v_1,\dots, v_{k+1}$ of vertices in an edge-ordered graph an \emph{increasing trail} of length $k$ if $v_iv_{i+1}$ form a strictly increasing sequence of edges for $1\le i \le k$. If all the vertices $v_i$ are distinct we call it an \emph{increasing path} of length $k$. 

\vspace{1mm}

Chv\'atal and Koml\'os \cite{CK1971} asked for the length of the longest increasing trail that one can guarantee in any edge-ordering of the complete $n$-vertex graph $K_n$. This question was solved by Graham and Kleitman in \cite{GK1973}. In the same paper \cite{CK1971}, Chv\'atal and Koml\'os also asked the corresponding question for a \textit{path} (rather than a trail). More precisely, they asked: What is the maximum integer $k$ such that every edge-ordering of $K_n$ has an increasing path of length $k$? It is very natural to ask this question for arbitrary host graphs (rather than just for complete graphs): The \emph{altitude} of a simple graph $G$ is defined as the maximum $k$ such that every edge-ordering of $G$ has an increasing path of length $k$. This seemingly simple question turned out to be quite challenging. Let $P_{k+1}^\mathrm{inc}$ denote the increasing path of length $k$. The maximal number of edges in a graph on $n$ vertices with a given altitude $k$ can have is precisely $\lex(n,P_{k+2}^\mathrm{inc})$. 

\vspace{1mm}

R\"odl \cite{R1973, W} proved that any graph $G$ with average degree $d\ge k(k+1)$ has altitude at least $k$. In other words, $\lex(n,P_k^\mathrm{inc})<\binom{k}{2}n$. (On the other hand, $\lex(n,P_k^\mathrm{inc})\ge\ex(n,P_k)=\frac{k-2}{2}n-O(k^2)$.) For sufficiently dense graphs, Milans \cite{M2017} proved that any graph $G$ with average degree $d$ has altitude at least $\Omega(d/(n^{1/3}(\log n)^{2/3}))$, where $n$ is the number of vertices in $G$. Very recently, Buci\'c, Kwan, Pokrovskiy, Sudakov, Tran, Wagner \cite{BKPSTW2018} significantly improved this bound, showing that the altitude is almost as large as $d$, provided $d$ is not too small. This result is close to being optimal because the longest path in a graph $G$ with average degree $d$ may be as short as $d$ (for example, if $G$ is a disjoint union of cliques of size $d + 1$). 
Inspired by the question of Chv\'atal and Koml\'os, several authors studied the altitude of various special classes of graphs including the hypercube \cite{DMPRT2016}, the random graph \cite{DMPRT2016,LL2016} and the convex geometric graph \cite{DFHP2017}. 
%this has been improved in \cite{BKPSTW2018} to $k\ge \frac{d}{2^{O(\sqrt{\log d \log\log n})}}$, which is better unless $d\ll n$.

\vspace{1mm}

Concerning the case when the forbidden edge-ordered graph is not a path (or a trail), a preliminary result was shown by Gerbner, Patk\'os and Vizer, \cite{GPV2017}, who proved that $\lex(n, C_4^{1243})=O(n^{5/3})$ and applied it to a problem in extremal set theory. Another interesting result is an unpublished result of Leeb (see the paper of Ne\v{s}et\v{r}il and R\"{o}dl \cite{NS}), stating that for any given $n \in \mathbb N$, every large enough edge-ordered complete graph contains a copy of $K_n$ such that the edges of this copy induce one of four special edge-orderings (see Section \ref{sec:canonical} for more details). Ramsey numbers of edge-ordered graphs have been studied recently (motivated by this paper), see \cite{BV2019,FL2019}. 

%Studying other parameters of edge-ordered graphs could also be interesting. 
%Other possible future directions include edge-ordered hypergraphs, or when the edges of the excluded graph are not linearly ordered, but form some other poset (if this poset is an antichain, we get back the classical Tur\'an number).

\subsection{Outline of the paper and main results}

In Section \ref{ESSorderchromaticSection} we present the analogue of Erd\H os-Stone-Simonovits theorem that applies to edge-ordered graphs. This theorem ties the Tur\'an number of an edge-ordered graph to its \emph{order chromatic number}, a notion that we will introduce. Order chromatic number is in turn strongly connected to some special edge-orders called \emph{canonical edge-orders}. This connection is discussed in Section \ref{sec:canonical}, where we also prove several important properties of order chromatic number (see, for example, Theorem \ref{canon} and Corollary \ref{corcanon}). In particular, it turns out that the order chromatic number behaves rather differently compared to the usual chromatic number in several aspects. For example the order chromatic number of a family of edge-ordered graphs can be substantially smaller than that of any single member of the family, and the order chromatic number of a finite edge-ordered graph can be infinite. In Section \ref{sec:diamond} we consider edge-ordered graphs with finite order chromatic number and estimate how large the order chromatic number can be in this case. Among other things, we will show that even when the order chromatic number of an edge-ordered graph $G$ is finite, it can grow exponentially in the number of vertices of $G$ (see Theorem \ref{explower}). Finally, in Section \ref{sec:bestandworst}, we briefly study the smallest and the largest possible order chromatic number of a graph $G$ over all possible edge-orderings of $G$. For most graphs, this latter number is infinite as shown by Theorem \ref{chiplus}. 

\vspace{1mm}

In Section \ref{sec:galaxy}, we study 
Tur\'an numbers of edge-ordered star forests. Recall that a {\em star} is a simple, connected graph in which all edges share a common vertex and a {\em star forest} is a non-empty graph whose connected components are all stars. We show a strong connection between this problem and Davenport-Schinzel theory and prove that the Tur\'an number is close to being linear for any given edge-ordered star forest (see Corollary \ref{linear}).

\vspace{1mm}

In Section \ref{sec:paths} we study Tur\'an numbers of edge-ordered paths. For edge-ordered paths with three edges, we determine the Tur\'an number exactly or up to an additive constant in Section \ref{sec:threeedgepaths}. And for most edge-ordered paths with four edges, in Section \ref{sec:fouredge} we show that the Tur\'an number is either $\Theta(n)$, $\Theta(n \log n)$ or $\Theta(n^2)$ (see the table at the beginning of Section \ref{sec:fouredge} for a complete list of results). This section also makes connections to the theory of forbidden submatrices. 

\vspace{1mm}

In Section \ref{sec:Fourcycle} we study Tur\'an numbers of edge-ordered 4-cycles. The 4-cycle $C_4$ has three non-isomorphic edge-orderings. The most interesting one is $C_4^{1243}$. For this one, using a special weighting argument, we show that the answer is close to $\Theta(n^{3/2})$ (as in the case of the usual Tur\'an problem). It is easy to show that the Tur\'an number of the other two edge-orderings of $C_4$ is $\binom{n}{2}$. 

\vspace{1mm}

Lastly, in Section \ref{conc} we make some concluding remarks. Tur\'an theory for edge-ordered graphs is very likely to have applications in other areas. As an example, using one of our results, we show that the maximum number of unit distances among $n$ points in convex position in the plane is $O(n \log n)$, reproving a result of Edelsbrunner-Hajnal \cite{EH91}, and F\"uredi \cite{F1990} (see Section \ref{sec:application}). We finish the paper with some open problems in Section \ref{sec:openprob}.

\vspace{1mm}

Throughout the paper, we use $\log$ to denote the binary logarithm.

\section{Erd\H os-Stone-Simonovits theorem for edge-ordered graphs \\ and Order chromatic number}
\label{ESSorderchromaticSection}

The most general result in Tur\'an-type extremal graph theory is the Erd\H os-Stone-Simonovits theorem, stated below.
Note that when using asymptotic notation to estimate the Tur\'an numbers of families of graphs or edge-ordered graphs, we always consider the family to be fixed. In particular, the $o(1)$ term in the following theorem tends to zero as $n$ goes to infinity, for a fixed family $\cH$.

\begin{theorem}[Erd\H os-Stone-Simonovits theorem, \cite{ES1966,ES1946}]\label{originaless}
Let $\cH$ be a family of simple graphs and $r+1:=\min\{\chi(H):H\in\cH\}\ge2$. We have
$$\ex(n,\cH)=\left(1-\frac 1r+o(1)\right)\frac{n^2}2.$$
\end{theorem}

The lower bound is given by the Tur\'an graph $T(n,r)$, which is the complete $r$-partite graph with each part having size $\lfloor n/r\rfloor$ or $\lceil n/r\rceil$.
Here $\chi(H)$ stands for the \emph{chromatic number} of the graph $H$. The key to extending this result to edge-ordered graphs is to find the notion that can play the role of chromatic number in the original theorem. We do this as follows.

\begin{defn}
We say that a simple graph $G$ \emph{can avoid} a family $\cH$ of edge-ordered graphs, if there is labeling $G^L$ of $G$ that avoids $\cH$. (In other words, note that a graph $G$ \emph{cannot avoid} $\cH$ if every labeling $G^L$ of $G$ contains a member of $\cH$.)

Let $\chi_<'(\cH)$, the \emph{order chromatic number} of $\cH$ stand for the smallest chromatic number $\chi(G)$ of a finite graph $G$ that cannot avoid $\cH$. In case all finite simple graphs can avoid $\cH$ we define $\chi_<'(\cH)=\infty$. In case the family $\cH$ contains a single edge-ordered graph we write $\chi_<'(H)$ to denote $\chi_<'(\{H\})$.
\end{defn}

\begin{remark}
Recall that when speaking of a family of edge-ordered graphs we assume no member of the family is empty. This makes the order chromatic number at least $2$.

We consider only finite graphs and edge-ordered graphs in this paper, so all members of $\cH$ are finite and so is $G$. But here we remark that the definition of the order chromatic number would not be altered if we allowed for infinite graphs $G$ -- this can be shown using compactness.

The notation $\chi_<'$ is usually reserved for the chromatic index of graphs but for notational convenience, in this paper, we use it to denote the order chromatic number. As the notion of chromatic index does not appear in this paper, we trust that this will not cause any confusion.

%We hope this will not cause any misunderstanding. We are not using chromatic index in this paper and using an alternate notation for the order chromatic number would be more cumbersome.
%\tg{Well, I'm not sure about this one. Maybe you want to change it back to $\chi_<$ (also used for the interval chromatic number) or $\chi_<'$. Or even $\chi$.}
%\dom{I've originally proposed $\chi_{<'}$}.
\end{remark}
\begin{theorem}[Erd\H os-Stone-Simonovits theorem for edge-ordered graphs]\label{ess}
If $\chi_<'(\cH)=\infty$, then
$$\lex(n,\cH)=\binom n2.$$
If $\chi_<'(\cH)=r+1<\infty$, then
$$\lex(n,\cH)=\left(1-\frac1r+o(1)\right)\frac{n^2}2.$$
\end{theorem}

\begin{proof} Clearly, if the simple graph $G=(V,E)$ can avoid $\cH$, then $\lex(|V|,\cH)\ge|E|$. If $\chi_<'(\cH)=\infty$ (or just larger than $n$), then the complete graph $K_n$ can avoid $\cH$, and this proves the first statement.

The lower bound for the second statement can be proved similarly as the Tur\'an graph $T(n,r)$ with $n$ vertices and $r$ classes has $(1-\frac{1}{r})\frac{n^2}{2}-O(r^2)$ edges and it can avoid $\cH$.

For the upper bound in the second statement, let $F$ be a simple graph with minimum chromatic number $\chi(F)=r+1$ that cannot avoid $\cH$. Clearly, we have $\lex(n,\cH)\le\ex(n,F)$. The bound then follows from the Erd\H os-Stone-Simonovits theorem: $\ex(n,F)=(1-\frac{1}{r}+o(1))\frac{n^2}{2}$.
\end{proof}

In light of Theorem \ref{ess}, the Tur\'an number of an edge-ordered graph is precisely captured by its order chromatic number. In the following subsections we will prove several properties of this parameter. In particular, we will show that `order chromatic number' is strongly connected to the notion of canonical edge-orders, studied in the next subsection.  

\subsection{Canonical edge-orders}\label{sec:canonical}

The Erd\H os-Stone-Simonovits theorem (Theorem~\ref{originaless}) connects the classical Tur\'an number to the well-established notion of chromatic number, while the vertex-ordered version \cite{PT2006} of Erd\H os-Stone-Simonovits theorem is connected to \textit{interval chromatic number} (which is easy to compute). Theorem~\ref{ess} shows that the order chromatic number is the relevant parameter for the Tur\'an number of edge-ordered graphs, but this notion seems less accessible. In Theorem~\ref{canon} we give criteria to determine the order chromatic number of a family of edge-ordered graphs. To decide whether the order chromatic number is two (that is, whether the Tur\'an number is quadratic in $n$) is especially simple, see Corollary~\ref{corcanon}.

Let us also emphasize that Theorem~\ref{ess} relates the Tur\'an number of a family of edge-ordered graphs to the order chromatic number of the \emph{family}. 

This is in contrast to the original Erd\H os-Stone-Simonovits theorem (or the vertex-ordered graph version in \cite{PT2006}), that speaks of chromatic number (interval chromatic number) of a \emph{single graph} (a single vertex-ordered graph) and relates the Tur\'an number of a family of graphs (vertex-ordered graphs) to the least chromatic number (interval chromatic number) of a member of the family. As we will see, here there is a meaningful difference, because the order chromatic number of a family can be substantially smaller than that of any single member in the family, see Proposition~\ref{non-principal}. In the context of extremal hypergraph theory such families are called non-principal. More precisely, a family $\cH$ of $r$-uniform hypergraphs is called \emph{non-principal} if any $r$-uniform hypergraph avoiding the family contains an asymptotically smaller fraction of the hyperedges of a complete $r$-uniform hypergraph than hypergraphs avoiding just a single element of the family. Balogh, \cite{B2002}, found non-principal families of 3-uniform hypergraphs of finite size. Later Mubayi and Pikhurko, \cite{MP2008}, found a non-principal family of size two.

Four labelings of the complete graph $K_n$ play a special role. We call them the \emph{canonical labelings} of $K_n$. For this definition we assume that the vertex set of $K_n$ is $\{v_1,\dots,v_n\}$.

\medskip 

$\bullet$ \emph{min-labeling} of $K_n$: For $1\le i<j\le n$ the label of the edge 
$v_iv_j$ is $L_1(v_iv_j)=ni+j$.

\smallskip 

$\bullet$ \emph{max-labeling} of $K_n$: For $1\le i<j\le n$ the label of the edge 
$v_iv_j$ is $L_2(v_iv_j)=nj+i$.

\smallskip 

$\bullet$ \emph{inverse min-labeling} of $K_n$: For $1\le i<j\le n$ the label of the edge 
$v_iv_j$ is $L_3(v_iv_j)=ni-j$.

\smallskip 

$\bullet$ \emph{inverse max-labeling} of $K_n$: For $1\le i<j\le n$ the label of the edge 
$v_iv_j$ is $L_4(v_iv_j)=nj-i$.
\smallskip

We need a similar notion of canonical labeling (edge-ordering) for complete multi-partite graphs as well. Let us denote by $K_{k\times n}$ the complete $k$-partite graph on $k\times n$  vertices. We denote the vertices of $K_{k\times n}$ by $v_{i,j}$ with $1\le i\le k$, $1\le j\le n$. For $1\le i\le k$ we call the set $V_i=\{v_{i,j}\mid1\le j\le n\}$ a \emph{class} of vertices and two vertices in $K_{k\times n}$ are adjacent if and only if they belong to distinct classes.

Informally, we call an edge-ordering of $K_{k\times n}$ \emph{canonical} if
the order of two edges is determined by the classes of the vertices they
connect and in case some of these vertices belong to the same class, then also by the order of those vertices within that class. To make this definition more formal, we concentrate on the complete bipartite graphs induced by $V_{i_1}\cup V_{i_2}$  that we call the \emph{parts} of the $K_{k\times n}$. Here $1\le i_1\ne i_2\le k$. (Note the slightly unusual use of the word `part', which sometimes means a vertex class of  $K_{k\times n}$ but in the rest of this subsection we use it to mean a complete bipartite graph induced by $V_{i_1}\cup V_{i_2}$ for some $1\le i_1\ne i_2\le k$.)  

An edge-ordering of the part induced by $V_{i_1}\cup V_{i_2}$ is \emph{canonical} if it is the linear order induced by one of the following eight labelings:
\begin{itemize}
\item $L_1(v_{i_1,j_1}v_{i_2,j_2})=nj_1+j_2$
\item $L_2(v_{i_1,j_1}v_{i_2,j_2})=nj_1-j_2$
\item $L_3(v_{i_1,j_1}v_{i_2,j_2})=-nj_1+j_2$
\item $L_4(v_{i_1,j_1}v_{i_2,j_2})=-nj_1-j_2$
\item $L_5(v_{i_1,j_1}v_{i_2,j_2})=nj_2+j_1$
\item $L_6(v_{i_1,j_1}v_{i_2,j_2})=nj_2-j_1$
\item $L_7(v_{i_1,j_1}v_{i_2,j_2})=-nj_2+j_1$
\item $L_8(v_{i_1,j_1}v_{i_2,j_2})=-nj_2-j_1$
\end{itemize}

We say that a part of $G_{k\times n}$ \emph{precedes} another part in a edge-ordering of $G_{k\times n}$ if all edges in the former part come before all edges in the latter part.

We say that an edge-ordering of $K_{k\times n}$  \emph{interleaves} the distinct parts induced by $V_{i_1}\cup V_{i_2}$ and $V_{i_1}\cup V_{i_4}$  if one of these four conditions hold:
 \begin{itemize}
\item For all $1\le j_1,j_2,j_3,j_4\le n$ the edge $v_{i_1,j_1}v_{i_2,j_2}$ precedes the edge $v_{i_1,j_3}v_{i_4,j_4}$ if and only if $j_1<j_3$.
\item For all $1\le j_1,j_2,j_3,j_4\le n$ the edge $v_{i_1,j_1}v_{i_2,j_2}$ precedes the edge $v_{i_1,j_3}v_{i_4,j_4}$ if and only if $j_1\le j_3$.
\item For all $1\le j_1,j_2,j_3,j_4\le n$ the edge $v_{i_1,j_1}v_{i_2,j_2}$ precedes the edge $v_{i_1,j_3}v_{i_4,j_4}$ if and only if $j_1>j_3$.
\item For all $1\le j_1,j_2,j_3,j_4\le n$ the edge $v_{i_1,j_1}v_{i_2,j_2}$ precedes the edge $v_{i_1,j_3}v_{i_4,j_4}$ if and only if $j_1\ge j_3$.
\end{itemize}
Note that the parts induced by $V_{i_1}\cup V_{i_2}$ and $V_{i_3}\cup V_{i_4}$ are never interleaved if $i_1$, $i_2$, $i_3$ and $i_4$ are four distinct indices.

We say that the edge-order of $K_{k\times n}$ is \emph{canonical}, if it induces a canonical edge-ordering on all the parts and for every two distinct parts either one precedes the other or they are interleaved by the ordering.

Clearly, the choice of the canonical edge-orders on the individual parts of $K_{k\times n}$ and the choices for their pairwise behavior determines the relation of every pair of edges. Some combination of these choices do not actually yield a transitive relation, but those that yield a transitive relation, give rise to a canonical edge-order of $K_{k\times n}$. It is easy to see that the same choices yield canonical edge-orders independent of the value of $n$ as long as $n\ge3$, so the number of canonical edge-orders of $K_{k\times n}$ depends on $k$ only. (Note that this observation fails for $k\ge4$ and $n=2$, so we do not consider the complete multipartite graphs $K_{2\times2}$.) In the simplest case of $K_{2\times n}$, we have a single part only, so we have exactly eight canonical edge-orders. It is easy to see that all of them yield isomorphic edge-ordered graphs. We denote this edge-ordered complete bipartite graph by $K_{n,n}^{\mathrm{can}}$. For greater values of $k$, however, both the number of canonical edge-orders of $K_{k\times n}$ and the number of non-isomorphic edge-ordered graphs obtained is growing fast. For $k\ge2$ we have $\binom{k}{2}!\cdot8^{\binom{k}{2}}$ canonical edge-orders without an interleaved pair of parts yielding $\binom{k}{2}!8^{\binom{k}{2}}/(k!2^k)$ non-isomorphic edge-ordered graphs. For $k=3$ this is 3072 canonical edge-orders and 64 non-isomorphic edge-ordered graphs. Still for $k=3$ there are 768 additional canonical edge-orders with a pair of interleaved parts yielding 16 additional non-isomorphic edge-ordered graphs.

The following theorem connects order chromatic number with the notion of canonical edge-orders. The first part of this theorem is not new and it goes back to an unpublished result of Leeb (see \cite{NS}), but we keep its simple proof for completeness.

%as was brought to our attention by Tran Manh Tuan after the first version of this paper appeared on arXiv, but can be found explicitly in \cite{NS}, where they say that the argument goes back to Leeb (unpublished).
%For completeness, we include our proof, which is essentially the same as theirs.

\begin{theorem}\label{canon}
\begin{itemize}
\item The order chromatic number of a family $\cH$ of edge-ordered graphs is infinity if and only if one of the four canonical edge-orders of $K_n$ avoids $\cH$ for all $n$.
\item For any $k\ge2$ and family $\cH$ of edge-ordered graphs, we have $\chi_<'(\cH)>k$ if and only if for all $n$, one of the canonical edge-orders of $K_{k\times n}$ avoids $\cH$.
\end{itemize}
\end{theorem}

Note that if $\cH$ is finite, then the ``for all $n$'' requirement in both parts of this theorem can be equivalently replaced by setting $n$ to be the largest number of vertices of any member of $\cH$.

\begin{proof}
We start with the proof of the first claim of the theorem. If a canonical (or any) edge-order of $K_n$
avoids $\cH$, then $K_n$ and therefore any of its subgraphs can avoid
$\cH$. If this holds for all $n$, then all finite graphs can avoid $\cH$, so its order chromatic number is infinity. This proves the ``if'' part of the claim.

Assume now that the order chromatic number is infinity, therefore any graph can avoid $\cH$. Take an edge-ordering $K$ of $K_m$ that avoids $\cH$ and color the 4-subsets of its vertices according to the order of the six edges between these vertices. That is, for $1\le j_1<j_2<j_3<j_4\le m$ we color the set $\{v_{j_1},v_{j_2},v_{j_3},v_{j_4}\}$ by the order of the six edges $v_{j_a}v_{j_b}$ in the induced subgraph. This is a 720-coloring of the 4-subsets. Let us choose a monochromatic subset $\{v_{j_l}\mid1\le l\le n\}$ such that $j_1<j_2<\cdots<j_n$. By the Ramsey theorem we can do this for any fixed $n$ if we start with a large enough complete graph $K_m$. It is not hard to see that a monochromatic subset on $n\ge5$ vertices is only possible for the four colors corresponding to the canonical edge-orders of $K_4$ and for the four colors corresponding to the edge-ordered graphs obtained from the canonical edge-orders of $K_4$ after reversing the order of the vertices. Such a monochromatic set induces an edge-ordered graph isomorphic to a canonical edge-order of $K_n$. Being a subgraph of $K$ that avoids $\cH$ it also avoids $\cH$ proving the ``only if'' part of the first claim.

For the proof of the second claim assume a canonical (or any) edge-order of $K_{k\times n}$ avoids $\cH$, so $K_{k\times n}$ can avoid $\cH$. As any finite graph of chromatic number at most $k$ is a subgraph of $K_{k\times n}$ for an appropriate $n$, we find that it can also avoid $\cH$ proving the ``if'' part of the second claim.

Assume now that $\chi_<'(\cH)>k$, therefore $K_{k\times m}$ can avoid $\cH$ for any $m$. Let us fix an edge-ordering $K$ of $K_{k\times m}$ avoiding $\cH$. We color the 4-subsets $H=\{\{j_1,j_2,j_3,j_4\} \ : \ 1\le j_1<j_2<j_3<j_4\le m \}$ with the order of the edges between the $4k$ vertices in $H^*=\{v_{i,j_l}\mid1\le i\le k,1\le l\le 4\}$. There are $16\binom{k}{2}$ such edges, so we have $(16{\binom{k}{2}})!$ colors. Let us assume that the subset $S$ formed by $j_1<j_2<\cdots<j_{kn}$ is monochromatic. By the Ramsey theorem we can find such a set for any $n$ if we start with a large enough value of $m$. Now we consider the edge-ordered subgraph $G$ of $K$ induced by the vertices $v_{i,j_l}$ for $1\le i\le k$ and $(i-1)n<l\le in$. Clearly, the underlying simple graph of $G$ is isomorphic to $K_{k\times n}$ with the isomorphism mapping $v_{i,l}$ of $K_{k\times n}$ to $v_{i,j_{(i-1)k+l}}$ in $G$. This isomorphism induces an edge-order on $K_{k\times n}$ and the fact that $S$ is monochromatic implies that this edge-order is canonical if $nk\ge5$. Indeed, for any pair of edges in $G$ their order is determined by the color of any set $H$ with $H^*$ containing all four endpoints, and thus by the common color of all 4-subsets of $S$. In particular, the order between two edges of $K_{k\times n}$ whose endpoints are in four distinct classes is determined by these classes and the order between two edges whose endpoints are in fewer classes is determined by the classes and the relative order of the endpoints in the common classes. The requirement that $nk\ge5$ is needed to ensure that if a subset is monochromatic with respect to our coloring of 4-subsets, then the same subset is also monochromatic with respect to a similar coloring of the 3-subsets.

Since $G$ is a subgraph of $K$, $G$ avoids $\cH$, so the canonical edge-order of $K_{k\times n}$ isomorphic to $G$ also avoids $\cH$ proving the ``only if'' part of the second claim and finishing the proof of the theorem.
\end{proof}

Theorem \ref{canon} implies the following corollary. 

\begin{corollary}\label{corcanon}
Let $\cH$ be a family of edge-ordered graphs.
\begin{enumerate}
\item For a subfamily $\cH'\subseteq\cH$ we have $\chi_<'(\cH')\ge\chi_<'(\cH)$.
\item We have $\chi_<'(\cH)=2$ if and only if there exists $G\in\cH$ with $\chi_<'(G)=2$.
\item For an edge-ordered graph $G$ on $n$ vertices we have $\chi_<'(G)=2$ if and only if $K_{n,n}^{\mathrm{can}}$ contains $G$.
\item If $\chi_<'(\cH)$ is finite, then there exists a subfamily $\cH'\subseteq\cH$ of size at most four with $\chi_<'(\cH')$ finite.
\item For $k\ge3$ there exists a number $c_k$ (depending only on $k$) such that if $\chi_<'(\cH)=k$, then there exists a subfamily $\cH'\subseteq\cH$ of size at most $c_k$ with $\chi_<'(\cH')=k$. One can choose $c_3=80$.
\end{enumerate}
\end{corollary}

\begin{proof}
The monotonicity claimed in part 1 of the corollary follows directly from the definition of the order chromatic number: if a graph can avoid a family $\cH$, then it can also avoid all its subfamilies.

For part 4 we use the first claim of Theorem~\ref{canon}. As the order chromatic number of $\cH$ is finite, none of the four canonical edge-orders of $K_n$ avoid $\cH$ for all $n$. For each one of the four canonical edge-orders, we can find a value of $n$ and an element $H$ of $\cH$ such that $K_n$ with that particular canonical edge-order does not avoid $H$. By the first part of Theorem~\ref{canon} again, the subfamily consisting of these four elements of $\cH$ has finite order chromatic number.

For part 5 we argue very similarly. Let $c_k$ be the number of canonical edge-orders of $K_{k\times n}$. By the second part of Theorem~\ref{canon} for each of the canonical edge-orders there is a choice of $n$ such that $K_{k\times n}$ with that edge-order contains a particular element $H$ of $\cH$. Let $\cH'$ consist of the elements of $\cH$ selected for one of those canonical edge-orders. By Theorem~\ref{canon} the order chromatic number of $\cH'$ is at most $k$. But by part 1 above it is at least $k$, so we have $\chi_<'(\cH')=k$.

Note that in this argument we could set $c_k$ to be the number of non-isomorphic edge-ordered graphs obtained from canonical edge-orderings of $K_{k\times n}$ as isomorphic edge-ordered graphs avoid the same edge-ordered graphs.

This makes us able to choose $c_3=80$ as claimed in part~5 and also proves parts 2 and 3 of the corollary as we know that each canonical edge-order of $K_{2\times n}$ is isomorphic to $K_{n,n}^{\mathrm{can}}$.
\end{proof}

The following two simple observations are related to $K_{n,n}^{\mathrm{can}}$ and part~3 of Corollary~\ref{corcanon}.

\begin{proposition}\label{obsbipar}
If $m$ is large enough compared to $n$, then $K_{m,m}$ cannot avoid $K_{n,n}^{\mathrm{can}}$.
\end{proposition}

\begin{proof}
$K_{2n,2n}^{\mathrm{can}}$ contains $K_{n,n}^{\mathrm{can}}$, so by part~3 of Corollary~\ref{corcanon} we have $\chi_<'(K_{n,n}^{\mathrm{can}})=2$. By definition, this implies the existence of a bipartite graph $G$ that cannot avoid $K_{n,n}^{\mathrm{can}}$. Clearly, $G$ is a subgraph of $K_{m,m}$ if $m$ is large enough, so neither can $K_{m,m}$ avoid $K_{n,n}^{\mathrm{can}}$.
\end{proof}

Note that Proposition~\ref{obsbipar} is also the two dimensional special case of a 1993 result by Fishburn and Graham, \cite{FG1993}. Recently Buci\'c, Sudakov and Tran, \cite{BST2019} proved that the choice $m=2^{2^{(4+o(1))n^2}}$ is enough for the statement of the proposition to hold. Note that this bound is much lower than the one that follows from the argument above or the result of Fishburn and Graham.

\begin{defn}
\label{closedefine}
We call a vertex $v$ of an edge-ordered graph \emph{close} if the edges incident to $v$ are consecutive in the edge-ordering, that is, they form an interval in the edge-order.
\end{defn}

\begin{proposition}\label{close}
If an edge-ordered graph $G$ is contained in the max-labeling or inverse max-labeling of some complete graph, then one of the end vertices of the maximal edge in $G$ is close in $G$. Symmetrically, if $G$ is contained in the min-labeling or inverse min-labeling of some complete graph, then one of the end vertices of the minimal edge in $G$ is close in $G$.

If $\chi_<'(G^L)=2$ for some labeling of a simple graph $G$, then $G$ has a
proper 2-coloring with all vertices in one color class being close in $G^L$. The converse also holds if $G$ is a forest, namely if the forest $G$ has a labeling $L$ and a proper 2-coloring in which all vertices of one of the color classes are close in $G^L$, then $\chi_<'(G^L)=2$.
\end{proposition}

Note that the requirement of $G$ being a forest is necessary in the last statement. The edge-ordered cycle $C_4^{1234}$ is bipartite, and all but one of its vertices are close, yet it is not contained in either the min-labeling or max-labeling of any complete graph, so $\chi_<'(C_4^{1234})=\infty$.

\begin{proof}
Consider any subgraph $G$ of $K_n$ with either the max-labeling or the inverse max-labeling. The larger-indexed end vertex of the maximal edge in $G$ is close in $G$. Similarly, in a subgraph $G$ of the min-labeling or inverse min-labeling of $K_n$ the smaller-indexed end vertex of the minimal edge in $G$ is close in $G$. This proves the first two statements of the proposition.

By Corollary~\ref{corcanon} we have $\chi_<'(G^L)=2$ if and only if $G^L$ is contained in $K_{n,n}^{\mathrm{can}}$ for some $n$. We think of $K_{n,n}^{\mathrm{can}}$ as the complete bipartite graph on vertices $u_i$ and $v_i$ with $1\le i\le n$ and with the label of edge $u_iv_j$ being $ni+j$. Clearly, the vertices $u_i$ are close in this graph and they form one color class of the only bipartition of $K_{n,n}$. These vertices remain close in any subgraph of $K_{n,n}^{\mathrm{can}}$ proving the third statement of the observation.

For the final statement we need to embed $G^L$ isomorphically to $K_{n,n}^{\mathrm{can}}$, where $n$ is the number of vertices in $G$. Let us fix a proper 2-coloring of $G$ with all vertices in one color class (say red)  being close. The red vertices are linearly ordered by the labeling of the incident edges (except for isolated vertices). We map the red vertices to vertices $u_i$ respecting this ordering, that is, if for red vertices $x$ and $y$ the edges incident to $x$ are lower than those incident to $y$, then we map $x$ to $u_{i_x}$ and $y$ to $u_{i_y}$ with $i_x<i_y$. Isolated red vertices can be mapped to any remaining vertices $u_i$. To obtain an isomorphic embedding all we have to do is map the vertices of the other color class to vertices $v_j$ in $K_{n,n}^{\mathrm{can}}$ such that for any red vertex $x$ the mapping of its neighbors $y$ respect the order of the labels $L(xy)$. As $G$ is a forest these requirements do not form a directed cycle, so all can be satisfied simultaneously.
\end{proof}

We finish this section by highlighting two aspects of the order chromatic number not shared by either the ordinary chromatic number of simple graphs or the interval chromatic number of vertex ordered graphs. 

\vspace{2mm}

Firstly, we show that the order chromatic number of a family of edge-ordered graphs can indeed be smaller than that of any member. The jump we exhibit here is the largest allowed by Corollary~\ref{corcanon}.
Recall that $P_5^{1423}$ denotes the path on five vertices, say $a,b,c,d,e$, with its edges ordered as $ab<cd<de<bc$.
Similarly, $P_5^{2314}$ is the path on five vertices, with edges ordered as $cd<ab<bc<de$.

\begin{proposition}\label{non-principal}
We have $\chi_<'(P_5^{1423})=\chi_<'(P_5^{2314})=\infty$, but $\chi_<'(\{P_5^{1423},P_5^{2314}\})=3$.
\end{proposition}

\begin{proof}
Notice that reversing the edge-order in $P_5^{1423}$ yields an edge-ordered graph isomorphic to $P_5^{2314}$ (by reversing the vertices) giving a certain symmetry to the statements of the proposition.

Neither endpoint of the smallest edge in $P_5^{2314}$ is close, so by Proposition~\ref{close}, $P_5^{2314}$ is not contained in either the min-labeling or the inverse min-labeling of a complete graph. Symmetrically, the max-labeling and the inverse max-labeling avoid $P_5^{1423}$. This shows that both of these edge-ordered paths have order chromatic number infinity if considered separately.

One can observe that both the max labeling and the inverse max-labeling of $K_n$ contain $P_5^{2314}$ as long as $n\ge5$ and symmetrically the min-labeling and the inverse min-labeling of $K_n$ contain $P_5^{1423}$. By Theorem~\ref{canon} this proves that that the order chromatic number of the pair $\{P_5^{1423},P_5^{2314}\}$ is finite. We want to prove specifically that it is $3$. By part 2 of Corollary~\ref{corcanon} it cannot be 2, so we need only to show that it is at most 3. Instead of exhibiting an explicit 3-chromatic graph that cannot avoid the pair, we use Theorem~\ref{canon} again and show that all canonical edge-orders of $K_{3\times2}$ contain one of the two edge-ordered paths. Let us recall that the classes of $K_{3\times2}$ are the pairs $V_i=\{v_{i,1},v_{i,2}\}$ for $1\le i\le 3$ and the parts of $K_{3\times2}$ are the complete bipartite subgraphs induced by two classes. Notice that in any canonical edge-order of $K_{3\times2}$ (or $K_{3\times n}$ in general) there is a smallest part preceding the other two parts or there is a largest part that is preceded by the other two parts. In the former case we can find an isomorphic copy of $P_4^{231}$ in the smallest part and then we can extend it with an edge from another part to get an isomorphic copy of $P_4^{2314}$. In the latter case we find an isomorphic copy of $P_4^{423}$ in the largest part and extend it with an edge from another part to obtain an isomorphic copy of $P_5^{1423}$. Thus, no canonical edge-order of $K_{3\times2}$ avoids both $P_5^{1423}$ and $P_5^{2314}$. This finishes the proof of the proposition.
\end{proof}

Secondly, recall that the order chromatic number of finite edge-ordered graphs can be infinite. In Section~\ref{sec:diamond}, we will show the existence of edge-ordered graphs for which the order chromatic number is finite but significantly larger than its number of vertices (or even its number of edges). In particular,  we will construct edge-ordered graphs $D_n$ on $n$ vertices for which the order chromatic number is finite but it still grows exponentially with $n$ (see Theorem \ref{explower}).

%We give an example of an edge-ordered graph where the order chromatic number is finite but significantly larger than its number of vertices or even its number of edges. We postpone the proof to Section~\ref{sec:diamond}, where we make further remarks on how large the order chromatic number of edge-ordered graphs can be. Let $D_4$ stand for the \emph{diamond}, the edge-ordered graph on vertices $x$, $y$, $z$ and $t$ with edges $xy<xz<xt<yt<zt$. Note that the underlying graph of the diamond is $K_4$ minus one edge and up to isomorphism, the diamond is the only edge-ordering of this simple graph with a finite order chromatic number (see Proposition~\ref{gendiamond}).

%\begin{proposition}\label{diamond}
%The order chromatic number of the diamond satisfies
%$$10\le \chi_<'(D)<\infty.$$
%\end{proposition}

\subsection{How large can the order chromatic number be?}\label{sec:diamond}
We saw examples of rather small edge-ordered graphs with order chromatic number infinity. In fact, Theorem~\ref{chiplus} below claims that \emph{every} simple graph with more than $3$ edges that is not a star forest has such an edge-ordering.
Here we consider edge-ordered graphs with finite order chromatic number. More specifically, we ask how large the order chromatic number of an $n$-vertex edge-ordered graph (or of a family of edge-ordered graphs with at most $n$ vertices in each) can be if it is finite.

Let $\mathcal K_n$ stand for the family of the four canonical labelings of the complete graph $K_n$. By the Ramsey theoretic theorem of Leeb that appeared in the paper \cite{NS} and stated here as the first half of Theorem~\ref{canon}, there exists $m$ for any $n$ such that $K_m$ cannot avoid $\mathcal K_n$. Let $f_{\rm Leeb}(n)$ be the smallest integer $m$ with this property.

\begin{proposition}\label{maxorderchrom}
If a family $\mathcal H$ of edge-ordered graphs on at most $n$ vertices has finite order chromatic number, then we have
$$\chi_<'(\mathcal H)\le\chi_<'(\mathcal K_n)\le f_{\rm Leeb}(n).$$
\end{proposition}

\begin{proof}
By Theorem~\ref{canon} and as the order chromatic number of $\mathcal H$ is finite, none of the four canonical edge-orders of $K_m$ avoid $\mathcal H$ for every $m$. But then none of the four edge-ordered graphs in $\mathcal K_n$ avoid $\mathcal H$ as otherwise the corresponding element of $\mathcal K_m$ would also avoid $\mathcal H$ for all $m$. (Note that we use here that all elements of $\mathcal H$ have at most $n$ vertices.)

The claim above implies that all edge-ordered graphs avoiding $\mathcal H$ also avoid $\mathcal K_n$ and therefore all simple graphs that cannot avoid $\mathcal K_n$ cannot avoid $\mathcal H$ either. This proves the first inequality.

To see the second inequality notice that for $m=f_{\rm Leeb}(n)$, $K_m$ cannot avoid $\mathcal K_n$ by definition, so we have $\chi_<'(\mathcal K_n)\le\chi(K_m)=m$.
\end{proof}

The upper bound on $f_{\rm Leeb}(n)$ coming from the argument presented in the proof Theorem~\ref{canon} is a Ramsey number for coloring 4-uniform hypergraphs of which the best available bound is triply exponential in $n$. However, a better upper bound that is doubly exponential in a polynomial of $n$ has been recently claimed by C. Reiher, V. R\"odl, M. Sales and M. Schacht , and, independently, also by D. Conlon, J. Fox and B. Sudakov; both proofs are unpublished at the moment \cite{schacht}. These recent results also contain a doubly exponential lower bound for $f_{\rm Leeb}(n)$. However, the lower bound does not seem to directly translate to a lower bound on $\chi_<'(\mathcal K_n)$ because it is possible that a very large graph with a small chromatic number cannot avoid the family $\mathcal K_n$.

\vspace{1mm}

Now we construct a sequence of edge-ordered graphs $D_n$ to show that the order chromatic number can grow exponentially in the number of vertices and still remain finite: For $n\ge 2$, let $D_n$ be the edge-ordered graph with vertices $x_1,\dots,x_n$ and the $2n-3$ edges incident to $x_1$ or $x_n$ with the edge-order $x_1x_2<x_1x_3<\cdots<x_1x_n<x_2x_n<\cdots<x_{n-1}x_n$. %Note that we called $D_4$ the \emph{diamond} in Proposition~\ref{diamond}.

\begin{proposition}\label{gendiamond}
Let $n\ge2$. We have $\chi_<'(D_n)<\infty$ but $\chi_<'(D^*)=\infty$ for every edge-ordering $D^*$ of the underlying simple graph of $D_n$ that is not isomorphic to $D_n$.
\end{proposition}

%Note that the $n=4$ special case of this proposition was stated in and just before Proposition~\ref{diamond}.

\begin{proof}
We show that $D_n$ is contained in all four canonical edge-orders of $K_n$. By Theorem~\ref{canon} this is enough to see that $\chi_<'(D_n)<\infty$.

We embed the vertices of $D_n$ in the min-labeled and the max-labeled $K_n$  in their natural order to show the containment. For the inverse min-labeled $K_n$ we use the order $x_1,x_n,x_{n-1},\ldots,x_2$. For the inverse max-labeled $K_n$ we use the order $x_{n-1},x_{n-2},\ldots,x_1,x_n$.

Now let $D^*$ be an edge-ordering of the underlying graph of $D_n$ with $\chi_<'(D^*)<\infty$. We need to show that $D^*$ is isomorphic to $D_n$.

For $n\le3$ the underlying graph of $D_n$ is $K_n$ and all its edge-orderings are isomorphic to $D_n$. Assume next that $n=4$. By Proposition~\ref{close}, $D^*$ must have a close vertex incident to the maximal edge and another incident to the minimal edge. These two vertices could only be $x_2$ and $x_3$ or $x_1$ and $x_4$. The former case yields two non-isomorphic edge-orderings: $x_1x_2<x_2x_4<x_1x_4<x_1x_3<x_3x_4$ and $x_1x_2<x_2x_4<x_1x_4<x_3x_4<x_1x_3$. The first of these is avoided by the inverse min-labeling, while the second is 
avoided by the min-labeling, so both have infinite order chromatic number. There are also two non-isomorphic edge-orders in which  $x_1$ and $x_4$ are the close vertices incident to the minimal and maximal edges. One of them is the edge-ordering of $D_4$, while the other is $x_1x_2<x_1x_3<x_1x_4<x_3x_4<x_2x_4$, yielding an edge-ordered graph avoided by all four canonical clique-labelings, so the order chromatic number of this edge-ordered graph is also infinite.

Finally for $n>4$ we use that the subgraph of $D^*$ induced by $x_1$, $x_n$ and any two other vertices (being an edge-ordering of the underlying graph of $D_4$) must be isomorphic to $D_4$ or the order chromatic number of the subgraph, and hence $D^*$ itself is infinite. This implies that $D^*$ itself is isomorphic to $D_n$.
\end{proof}

To prove an exponential lower bound on $\chi_<'(D_n)$ (in Theorem \ref{explower}), we need the following lemma.

%, which is also used to prove Proposition~\ref{diamond}.

\begin{lemma}\label{dialemma}
Let $n\ge2$ and $m\ge2$ be integers. We have $\chi_<'(D_n)>m$ if and only if there is an edge-ordering $K$ of $K_m$ avoiding $D_n$ such that the auxiliary graphs $G_x$ are bipartite for all vertices $x$ of $K_m$. Here $V(G_x)=V(K_m)\setminus \{x\}$ and $E(G_x)$ consists of the edges $yz$ such that $xy<yz<zx$ in the edge-ordering of $K$.
\end{lemma}

\begin{proof}
By Theorem~\ref{canon} we have $\chi_<'(D_n)>m$ if and only if there is a canonical edge-ordering of $K_{m\times n}$ avoiding $D_n$. To prove the ``only if'' part of the lemma let us assume that $K^*$ is a canonical edge-ordering of $K_{m\times n}$ that avoids $D_n$. Recall that the vertices of $K_{m\times n}$ are $v_{i,j}$ with $1\le i\le m$ and $1\le j\le n$. The vertices $v_{i,j}$ with a fix $i$ form an independent set that we call a \emph{class}. The \emph{parts} of $K_{n\times m}$ are the complete bipartite graphs connecting two classes.

Let $K$ be the subgraph of $K^*$ induced by the $m$ vertices $v_{i,1}$. It is a labeling of the complete graph $K_m$. We claim that $K$ satisfies the conditions in the lemma, namely it also avoids $D_n$ and the auxiliary graphs $G_x$ are all bipartite.

As a subgraph of $K^*$, $K$ avoids $D_n$. We need to prove that the auxiliary graphs are bipartite. So let $x=v_{i,1}$ be a fixed vertex of $K$. Consider another vertex $y$ of $K$ and the order of the edges $yv_{i,j}$. As on $K^*$ the edge-order is canonical, this is either monotone increasing in $j$ or monotone decreasing in $j$. We call the vertex $y$ \emph{increasing} or \emph{decreasing} accordingly. We claim that $G_x$ is bipartite because all its edges connect an increasing vertex with a decreasing vertex. Assume for a contradiction $yz$ is an edge of $G_x$ with both $y$ and $z$ being increasing (or both being decreasing). Now consider the subgraph of $K^*$ induced by the vertices $y$, $z$ and $n-2$ of the vertices in the class of $x$. This subgraph is isomorphic to $D_n$ contradicting the assumption that $K^*$ avoids $D_n$. The contradiction proves the ``only if'' part of the lemma.

For the ``if'' part let $K$ be an edge-ordering of $K_m$ satisfying the conditions of the lemma. We need to find a canonical edge-ordering $K^*$ of $K_{m\times n}$ avoiding $D_n$. We identify the vertices of $K$ with the classes in $K_{m\times n}$. This way the parts of $K_{n\times m}$ correspond to the edges of $K$. In our canonical edge-ordering no pair of parts are interleaved and one part precedes another if the edge in $K$ corresponding to the former part is smaller than the edge corresponding to the latter part. Now we consider the bipartite auxiliary graphs $G_x$ and fix a bipartition to ``increasing'' and ``decreasing'' vertices. Note that the same vertex can be designated increasing in one auxiliary graph and decreasing in another. To specify the canonical edge-order of $K^*$ we have to further specify one of the eight canonical orders for each of the $\binom m2$ parts. Assume the classes $V_{i_1}=\{v_{i_1,j}\mid1\le j\le n\}$ and $V_{i_2}=\{v_{i_2,j}\mid1\le j\le n\}$ correspond to vertices $x$ and $y$ in $K$. We choose the canonical order of the part between these two classes such that the order of the edges $v_{i_1,j_1}v_{i_2,j_2}$ is increasing or decreasing in $j_2$ (for fixed $j_1$) according to whether $x$ is increasing or decreasing in $G_y$. Similarly, we choose the canonical order of the part of $K^*$ induced by $V_{i_1}\cup V_{i_2}$ such that the order of these edges is increasing or decreasing in $j_1$ (for a fixed $j_2$) according to whether $y$ is increasing or decreasing in $G_x$. For any of the four possible cases we still have two canonical edge-orders to choose from and we choose arbitrarily.

We claim that $K^*$ avoids $D_n$. Assume for a contradiction that $K^*$ contains $D_n$. If the $n$ vertices of the isomorphic copy of $D_n$ come from $n$ different classes in $K^*$, then this subgraph of $K^*$ would correspond to an isomorphic subgraph of $K$. This contradicts our assumption that $K$ avoids $D_n$.

Now assume that two vertices $a$ and $b$ of the subgraph of $K^*$ isomorphic to $D_n$ come from the same class. As the class is independent in $K^*$, $a$ and $b$ must correspond to two non-adjacent vertices in $D_n$. Let $c$ and $d$ be the vertices in the subgraph corresponding to the full degree first and last vertex of $D_n$. Clearly, $a$, $c$ and $d$ must come from three distinct classes of $K_{m\times n}$. Let $x$, $y$ and $z$ be the corresponding vertices in $K$. The subgraph of $K^*$  induced by the four vertices $a$, $b$, $c$ and $d$ must be isomorphic to $D_4$ and this implies that $yz$ is an edge in $G_x$. So one of $y$ and $z$ must be a decreasing vertex in $G_x$, the other an increasing vertex. But that means that either $ac<bc$ and $ad>bd$ in $K^*$ or vice versa: $ac>bc$ and $ad<bd$ in $K^*$. Both cases contradict the isomorphism of the induced subgraph to $D_4$. The contradiction finishes the proof of the lemma.
\end{proof}

Now we are ready to prove the exponential lower bound on $\chi_<'(D_n)$.

\begin{theorem}\label{explower}
$\chi_<'(D_n)> 2^{n-2}$.
\end{theorem}

\begin{proof}
Using Lemma~\ref{dialemma}, it is enough to give an edge-ordering $K$ of $K_{2^{n-2}}$ that avoids $D_n$ such that the auxiliary graphs $G_x$ are bipartite for all vertices $x$ of $K$.
Each vertex of $K$ will correspond to a binary sequence $\{0,1\}^{n-2}$ and we write $u<v$ if $u$ comes before $v$ in the lexicographic order.
It is convenient to think about these sequences as root to leaf paths in a complete binary tree of depth $n-2$, which is drawn in the ``usual way'', i.e., its leaves are on a line in lexicographic order.

The order on the edges of $K$ is defined as an extension of the following acyclic relation.

\begin{itemize}
    \item $(u,v) < (u,w)$ if $u$ differs first from $v$ at an earlier place than from $w$, i.e., the paths $u$ and $w$ go the same way where $v$ splits from $u$.
    E.g., $(00,11)<(00,01)$.
    \item $(u,v) < (u,w)$ if $u$ differs first from $v$ at the same place as from $w$, and $v$ is closer to $u$ than $w$ (as ordered leaves in the binary tree), so either $u<v<w$, or $u>v>w$.
    E.g., $(10,01) < (10,00)=(00,10) < (00,11)$.
\end{itemize}

This is indeed acyclic and we can take any linear extension of it to get an edge-ordering on $K$.
Another way to define an extension of the above relation would be to say that the place of an edge $xy$ in the order is determined first by the length of the longest common prefix of $x$ and $y$, or in case these prefixes are of the same length, then the edges are ordered by the value $|x-y|$ (where $x$ and $y$ are understood as binary numbers), finally edges having a tie in both values are ordered arbitrarily.
This gives a linear order right away.

Now we need to prove that such a linear order avoids $D_n$ and each $G_x$ is bipartite.

The latter claim is easy; $(x,u) < (u,v) < (x,v)$ can happen only if $u$ and $v$ are on different sides of $x$, thus the two classes of $G_x$ will be $\{u\mid u<x\}$ and $\{v\mid v>x\}$.

Now suppose that $K$ contains a $D_n$ satisfying 
$x_1x_2<x_1x_3<\cdots<x_1x_n<x_2x_n<\cdots<x_{n-1}x_n$.
This is only possible if $x_1$ differs first from $x_2,\ldots,x_n$ at the same place, otherwise we would have $x_1x_i$ and $x_nx_i$ both smaller or both larger than $x_1x_n$ for some $1<i<n$.
The order of the edges of $x_1$ implies that $x_1<x_2<\cdots<x_n$ or $x_1>x_2>\cdots>x_n$.
If some $1<i<j<n$, $x_i$ and $x_j$ differ first from $x_n$ at the same place, then we would have $x_ix_n>x_jx_n$, which is not the case in $D_n$.
Thus for $i<n$ each $x_i$ differs at a different place from $x_n$, but this is not possible because of the pigeonhole principle.
\end{proof}

Note that the bound of Theorem \ref{explower} is trivially sharp if $n\le 3$, but the proposition below shows that it is not for $n=4$.

\begin{proposition}\label{diamond}
The order chromatic number of the diamond satisfies
$$10\le \chi_<'(D_4)<\infty.$$
\end{proposition}

\begin{proof}[Proof of Proposition~\ref{diamond}.]
We have already seen in Proposition~\ref{gendiamond} that $\chi_<'(D_4)$ is finite. %Proposition~\ref{ramsey} even gave an upper bound, albeit a rather large one.

For the lower bound we use Lemma~\ref{dialemma}. The edge labeling $K_9^L$ satisfies the conditions of that lemma, where the vertices of $K_9$ are $v_1,v_2,\dots,v_9$ and the label of the edge $v_iv_j$ is given as the $j$'th entry in the $i$'th row of the following symmetric matrix. The entries in the diagonal are left blank. A short case analysis is enough to verify that $K_9^L$ satisfies the properties required in Lemma~\ref{dialemma}, but we found the labeling itself by computer search.%\footnote{Our code ran in 200 seconds on a 53-qubit computer. The calculations were independently verified by a classical supercomputer in two and a half days.}
% $$\begin{pmatrix}
% &25&3&9&24&11&22&23&31\\
% 25&&1&8&26&29&4&27&30\\
% 3&1&&36&2&5&34&35&33\\
% 9&8&36&&17&16&14&19&6\\
% 24&26&2&17&&15&20&18&7\\
% 11&29&5&16&15&&12&13&10\\
% 22&4&34&14&20&12&&21&32\\
% 23&27&35&19&18&13&21&&28\\
% 31&30&33&6&7&10&32&28&
% \end{pmatrix}$$

$$\begin{pmatrix}
&1&2&3&4&33&34&35&36\\
1&&26&25&29&30&5&27&8\\
2&26&&24&15&7&20&18&17\\
3&25&24&&11&31&22&23&9\\
4&29&15&11&&10&12&13&16\\
33&30&7&31&10&&32&28&6\\
34&5&20&22&12&32&&21&14\\
35&27&18&23&13&28&21&&19\\
36&8&17&9&16&6&14&19&
\end{pmatrix}$$
\end{proof}

With some simple observations combined with a computer search we could also prove $\chi_<'(D_4)\le 31$.
More precisely, we checked by computer that the edges of a $K_6$ on $v_1,\ldots v_6$ cannot be ordered to avoid $D_4$ if $v_0v_i<v_0v_j$ for $i<j$, %dom: without this there is example
and for each $i$ the bipartite $G_{v_i}$ is empty on $\{v_j\mid j>i\}$, and this can always be guaranteed for some $6$ vertices of $K_{32}$ using that each $G_x$ has an independent set of size $16$, since $G_x$ is bipartite.
With similar tricks the bound $31$ can be certainly reduced, but proving the conjectured $\chi_<'(D_4)=10$ is out of reach with such methods.

A strongly related question is that whether in a vertex- and edge-ordered complete graph on $N=2^n$ vertices, can we always find an $n=\log N$ vertex subgraph where any vertex's left going edges are all smaller than its right going edges, or vice versa? Note that $N\le R_3(n)\le 2^{2^n}$ follows from a simple Ramsey argument: %similar to the one in Proposition \ref{ramsey}:
just two-color each triple $u<v<w$ depending on whether $uv<vw$ or not.

%Dom: oh well, this attempt failed, so I've removed it...
% \subsubsection{NEW}
% Here we prove an exponential upper bound on $\chi_<'(D_n)$.
% Using Lemma~\ref{dialemma}, it is enough to show that for some $N=2^{poly(n)}$ no edge-ordering of $K_N$ such that the auxiliary graphs $G_x$ are bipartite for all vertices $x$ of $K_N$ can avoid $D_n$.
% Order the vertices of $K_N$ arbitrarily.
% We will throw away some of the vertices now.
% We keep $v_1$ and $v_2$.
% But we throw away at most half of the vertices $\{v_j\mid j>2\}$, in such a way that for all the remaining vertices $v_1v_2<v_2v_j$ or for all $v_1v_2>v_2v_j$.
% We will maintain the property that after step $i$, we have that for any two kept vertices $v_a,v_b$ with $a,b\le i$ the edges $v_bv_j$ are either all larger, or all smaller than $v_av_b$.
% After $m$ steps, all vertices will satisfy this property if $N>m2^m$.

% Since for any $v_i$ the auxiliary graph $G_{v_i}$ is bipartite, we can throw away at most half of the vertices $\{v_j\mid j<i\}$ and $\{v_j\mid j>i\}$ such that in the auxiliary graph $G_{v_i}'$ of the remaining graph is a complete bipartite graph.
% In fact, we will use this for the smallest vertex, so will get that $G_{v_1}$ is empty.

\subsection{The best and worst edge-orders of a graph}
\label{sec:bestandworst}

%\tg{Most of the results in this section can be formulated via order chromatic number, so that is how I wrote them. The claim that all subquadratic Tur\'an numbers are $O(n^{2-\varepsilon})$ is a general statement not only for $\lex^+$, so if we want to write it this is not the space. On star forests and $P_3$ we say more later anyway. It leaves the remark on the exact Tur\'an number of the edge-ordered $K_3$. We could add that somewhere too.}

For a non-empty finite graph $G$, let $\chi^-(G)=\min_L\chi_<'(G^L)$, where the minimum is taken over all labelings $L$ of $G$. Similarly, let $\chi^+(G)=\max_L \chi_<'(G^L)$. In this subsection we determine $\chi^+(G)$ for all graphs $G$, and prove the following simple result concerning $\chi^-(G)$.

%and $\chi^-(G)$ for many simple graphs $G$ and $\chi^+(G)$ for all of them.

\begin{proposition}\label{bip-}
$\chi^-(G)\ge\chi(G)$ for any graph $G$.

$\chi^-(G)=2$ if and only if $\chi(G)=2$.
\end{proposition}

\begin{proof} If a graph $H$ does not contain the graph $G$ as a subgraph, then clearly all labelings of $H$ avoid all labelings of $G$. This proves the first statement and also the only if part of the second statement.

If $G$ is bipartite, then it is contained in the complete bipartite graph $K_{n,n}$ for an appropriate $n$. The canonical labeling $K_{n,n}^{\mathrm{can}}$ induces a labeling $G^L$ of $G$ that is contained in $K_{n,n}^{\mathrm{can}}$, so by Corollary~\ref{corcanon} we have $\chi_<'(G^L)=2$. This finishes the proof of the proposition. 
\end{proof}

One might think that $\chi^-(G)$ is always finite.  However, this is not the case as the following proposition shows.

\begin{proposition}\label{K4} $\chi^-(K_4)=\infty$.
\end{proposition}

%\tg{Here I killed $B_3$, because I don't think the claim is true. I think even $B_k$ ($K_{2,k}$ plus the edge $u_1u_2$ connecting the vertices on the two vertex side) has a labeling with finite chromatic number: $L(u_1v_j)=j$, $L(u_1u_2)=k+1$ and $L(u_2v_j)=k+j+1$. Am I missing something? On the other hand I found a graph similar to $B_3$ with $\chi^-=\infty$. It's also obtained from $K_{2,3}$ by adding an extra edge but we add the edge on the other side. Should we include this?}

\begin{proof} Consider a labeling $L$ of $K_4$. If the three largest edges of $K_4^L$ do not form a star, then neither endpoint of the largest edge in $K_4^L$ is close, so $K_4^L$ is not contained in the max-labeling or inverse max-labeling of a complete graph by Proposition~\ref{close}. If the three largest edges in $K_4^L$ do form a star, then the three smallest ones do not form a star, so (again by Proposition~\ref{close}) $K_4^L$ is not contained in the min-labeling or inverse min-labeling of a complete graph. Therefore, an appropriate edge-ordering of $K_n$ always avoids $K_4^L$. This finishes the proof.

A closer inspection reveals that every labeling of $K_4$ is avoided by at least three of the four canonical labelings of $K_n$. Indeed, a subgraph induced by four vertices of a canonical labeling of $K_n$ is always isomorphic to the corresponding canonical labeling of $K_4$ and the four canonical labelings of $K_4$ are pairwise non-isomorphic.
\end{proof}

We call a simple non-empty graph a \emph{star forest} if all connected components are stars. We will study the Tur\'an numbers of edge-ordered star forests in more detail in the next section. As isolated vertices do not affect the order chromatic number we only consider simple graphs without isolated vertices.

\begin{theorem}\label{chiplus}
We have $\chi^+(G)=2$ if $G$ is a star forest or $G=P_4$. We have $\chi^+(K_3)=3$. For all remaining finite simple graphs $G$ without isolated vertices we have $\chi^+(G)=\infty$.
\end{theorem}

\begin{proof} We prove the first statement using Proposition~\ref{close}. Any star forest has a proper 2-coloring with all vertices in one color class having degree one. These vertices are close in all labelings. Both color classes of $P_4$ contain a degree 2 vertex, but at any edge-ordering of $P_4$ makes one of them close, so the last statement of Proposition~\ref{close} applies again.

$K_3$ is not bipartite, so $\chi_<'(K_3^L)\ge3$ for all labelings $L$. But all labelings of $K_3$ yield isomorphic edge-ordered graphs, so $K_3$ cannot avoid $K_3^L$ for any $L$. This makes $\chi^-(K_3)=\chi^+(K_3)=3$.

Any remaining non-empty graph $G$ without an isolated vertex contains an edge $e_1=uv$ such that both $u$ and $v$ have degree more than $1$. We find a labeling of $G$ that is avoided by both the max-labeling and the inverse max-labeling of any complete graph by making $e_1$ the maximal edge and ensuring neither $u$ nor $v$ is close, see Proposition~\ref{close}. If there exists an edge not adjacent to $e_1$ we are done by making it the second largest. If all edges are adjacent to $e_1$, then one of $u$ or $v$ must have degree at least $3$ as $G$ has at least $4$ edges. Say $e_2$ and $e_3$ are both incident to $u$. Making $e_2$ the second largest we ensure $v$ is not close and making $e_3$ the smallest we ensure $u$ is not close either.
\end{proof}

Another natural question to study is how $\lex(n,G^L)$ behaves for the best and worst edge-orderings of a given graph $G$.
By Theorem \ref{originaless}, $\lex(n,G^L)$ is asymptotically determined by $\chi_<'(G^L)$ if $\chi_<'(G^L)>2$. Proposition \ref{K4} and Theorem \ref{chiplus} imply that for many graphs $\chi^-(G)=\chi^+(G)=\infty$, so even for the best edge-order, $\lex(n,G^L)=\binom n2$ because of Theorem \ref{canon}.
We have also seen in Section \ref{sec:diamond} that even $\chi^-(D_k)$ can grow exponentially in $k$, while $\chi(D_k)=3$.
In fact, if we denote by $K_{2,3}^+$ the graph obtained by adding an edge connecting two vertices on the larger side of $K_{2,3}$, then we have $\chi(K_{2,3}^+)=3$, but $\chi^-(K_{2,3}^+)=\infty$.
(This can be proved with a case analysis similar to the proof of Proposition \ref{K4}.)

Proposition \ref{bip-} shows that $\chi(G)=2$ implies $\chi^-(G)=2$.
Is it in fact possible that for every bipartite $G$ there an edge-ordering $L$ such that $\lex(n, G^L)= O(ex(n,G))$?
As we have discussed in the Introduction, this is true when $G$ is a path, because we can pick the monotone increasing edge-labeling for which
$\lex(n,P_{k}^\mathrm{inc})=O(n)$.
It, however, fails for most trees.

\begin{proposition}\label{prop:allLbad}
If a tree $T$ has a vertex from which $3$ paths of length $3$ start, then $\lex(n, T^L)=\Omega(n\log n)$ for any edge-ordering $T^L$ of $T$.
\end{proposition}
%Proof: consider the spider with 3 legs of length 3, i.e., a tree with one vertex of degree 3, from which 3 paths of length 3 start. Order the edges around the degree 3 vertex as a<b<c. WLOG and by symmetry, the third edge adjacent to 'b' is >b, say d. But then d>b<c forms a 'valley', and WLOG generality we can suppose d>c. Taking an edge adjacent to 'c' creates a path with d-b-c-x, but any such path has lex at least n log n, irrespective of what x is. Therefore, the lex for the tree is also at least n log n. 

The proof of Proposition \ref{prop:allLbad} follows from a simple case analysis which shows that such trees $T^L$ always contain a path $P_5$ of length $4$ such that the restriction of the edge-ordering of $T^L$ to this path, yields an edge-ordered path $P_5^L$ for which $\lex(n,P_5^L)=\Omega(n\log n)$.
(For the characterization of length $4$ paths, see Section \ref{sec:fouredge}.)

%left for future research to classify such trees, like caterpillars always O(n) etc.

\section{Star forests}\label{sec:galaxy}

\newcommand\ds{{\mathrm{ex}_{DS}}}

Recall that a {\em star} is a simple, connected graph in which all edges share a common vertex and a {\em star forest} is a non-empty graph whose connected components are all stars. In this section we study the Tur\'an numbers $\lex(n,F)$ for edge-ordered star forests $F$. We will show that this problem is closely related to Davenport-Schinzel theory, so let us recall the basic definitions. For a more thorough introduction on Davenport-Schinzel theory see e.g., \cite{K1992}.

A {\em word} is a finite sequence. We will refer the elements of the sequence as {\em letters}, but we are not interested in what the actual letters are, we only care about where the same letters repeat. Accordingly, we say that the words $u=a_1\dots a_n$ and $v=b_1\dots b_m$ are {\em equivalent} if $n=m$ and for all $1\le i,j\le n$ we have $a_i=a_j$ if and only if $b_i=b_j$. We denote the length of the word $u$ by $|u|$, so we have $|u|=n$ in this example. We write $||u||$ for the number of distinct letters in $u$. A word $u$ is {\em $k$-regular} (for some positive integer $k$) if every $k$ consecutive letters in $u$ are distinct (in case $|u|<k$ we require all letters of $u$ to be distinct). A {\em subword} is obtained by deleting any number of letters from a word and considering the word formed by the remaining letters in their original order. We say that a word $u$ {\em contains} another word $f$ if $f$ is equivalent to a subword of $u$. If this is not the case we say that $u$ {\em avoids} $f$. For a non-empty word $f$ and a positive integer $n$ we write $\ds(n,f)$ for the length $|u|$ of the longest $||f||$-regular word $u$ on at most $n$ letters (that is $||u||\le n$) avoiding $f$. The central problem of Davenport-Schinzel theory is to calculate or estimate this extremal function.

To apply the results of Davenport-Schinzel theory we need to relate edge-ordered graphs to words. We do this in two different ways. First, let $F$ be an edge-ordered star forest. We represent each component of $F$ with a unique letter. We define the corresponding word $w(F)$ to be $w(F)=a_1\dots a_m$, where $m$ is the number of edges in $F$ and $a_i$ is the letter representing the component of $F$ containing the $i$'th edge (in the edge-ordering of $F$). We obtain the longer word $w'(F)=a_1^{2m}\dots a_m^{2m}$ by repeating each letter in $w(F)$ $2m$ times. (Here we use exponentiation to denote repetitions.) For our second connection between graphs and words consider an arbitrary edge-ordered graph $G$. We build the corresponding word $u(G)$ over the set of vertices of $G$ as letters by listing the two end vertices of each edge. We list the edges according to their edge-order but we choose the order of the two end vertices of the same edge arbitrarily. For example, if $G$ is a graph with edges $ab, ac, bc, ad$ with the edge-order $ab < ac < bc < ad$, then $u(G)$ could be $abaccbda$. (Strictly speaking, $u(G)$ is therefore not well defined, but hopefully this ambiguity will cause no confusion.) The length of $u(G)$ is twice the number of edges in $G$.

The main connection between the containments in these two different contexts is provided by the following lemma.

\begin{lemma}\label{DScontain}
Let $F$ be an edge-ordered star forest and let $G$ be an edge-ordered graph. If $u(G)$ contains $w'(F)$, then $G$ contains $F$.
\end{lemma}

\begin{proof}
Let $w(F)=a_1\dots a_m$. Let $u$ be the subword of $u(G)$ equivalent to $w'(F)=a_1^{2m}\dots a_m^{2m}$. We have $u=b_1^{2m}\ldots b_m^{2m}$. Each of the letters in $u$ were inserted in $u(G)$ as an end vertex of an edge in $G$, thus $b_i^{2m}$ must come from $2m$ distinct edges of $G$, each incident to the vertex $b_i$. For each $i=1,\dots,m$ we select one of these edges, $e_i=b_ic_i$ such that the vertices $c_i$ are pairwise distinct and none of them coincides with any of the vertices $b_j$. We can achieve this (even in a greedy manner) as out of the $2m$ possibilities for the choice of $c_i$, less than $2m$ is forbidden.

It is easy to see that $F$ and the subgraph of $G$ consisting of the vertices $b_i$, $c_i$ (for $i=1,\dots,m$) and the edges $e_i$ for $i=1,\dots,m$ are isomorphic as edge-ordered graphs.
\end{proof}

Davenport-Schinzel theory bounds the length of the $||f||$-regular words avoiding a forbidden word $f$. We will use this bound for $f=w'(F)$ together with Lemma~\ref{DScontain} to bound the length of $u(G)$ (and with that the number of edges in $G$) for edge-ordered graphs $G$ avoiding the edge-ordered star forest $F$. The only obstacle here is that $u(G)$ does not have to be $||w'(F)||$-regular. In fact, it does not even have to be 2-regular. The next lemma helps us overcome this difficulty.

\begin{lemma}\label{kregular}
Let $k>1$ be an integer and let $G$ be an edge-ordered graph with $m$ edges. The word $u(G)$ has a $k$-regular subword of length larger than $m/(k-1)$.
\end{lemma}

\begin{proof} Recall that $u(G)=a_1a_2\dots a_{2m}$, where $a_{2i-1}a_{2i}$ is the $i$'th edge of $G$. We apply the following (standard) greedy procedure to obtain a $k$-regular subword. We start with the empty word $u_0$ and for $1\le i\le 2m$ define $u_i=u_{i-1}a_i$ if $u_{i-1}a_i$ is $k$-regular, or $u_i=u_{i-1}$ otherwise. Clearly $u=u_{2m}$ is a $k$-regular subword of $u(G)$.

Consider any edge $e=a_{2i-1}a_{2i}$ of $G$. Both of the endpoints $a_{2i-1}$, $a_{2i}$ must appear among the last $k$ letters of $u_{2i}$, either because we inserted $a_{2i-1}$ or $a_{2i}$ (or both) after $u_{2i-2}$ or because we did not insert them, so they were already among the last $k-1$ letters in $u_{2i-2}$. Thus $e$ connects two vertices that appear in $u$ at distance at most $k-1$ from each other. As there are fewer than $(k-1)|u|$ pairs of this type, we have $m<(k-1)|u|$ and $|u|>m/(k-1)$ as needed.
\end{proof}

\begin{theorem}\label{DSbound}
Let $F$ be an edge-ordered star forest with $k>1$ components. We have
$$\lex(n,F)\le (k-1)\ds(n,w'(F)).$$
\end{theorem}

\begin{proof}
Let $G$ be an edge-ordered graph with $n$ vertices and $m=\lex(n,F)$ edges that does not contain $F$. By Lemma~\ref{DScontain}, $u(G)$ avoids $w'(F)$. Any subword of $u(G)$ must also avoid $w'(F)$, among them the $k$-regular subword of length at least $m/(k-1)$ guaranteed by Lemma~\ref{kregular}. Note that $k=||w'(F)||$ and $||u(G)||\le n$. By the definition of the extremal function $\ds(n,w'(F))$ this means that $m/(k-1)\le\ds(n,w'(F))$ as required.
\end{proof}

We use this last theorem to prove an almost linear upper bound on $\lex(n,F)$ for an arbitrary edge-ordered star forest $F$ and linear upper bound for certain special edge-ordered star forests.

\begin{corollary}\label{linear}
Let $F$ be an edge-ordered star forest. We have
$$\lex(n,F)\le n2^{(\alpha(n))^c},$$
where $\alpha(n)$ is the extremely slow growing inverse Ackermann function and the exponent $c$ depends on $F$, but not on $n$.

Further, if $w(F)$ is of the form $a^ib^ja^kb^l$ for two distinct letters $a$ and $b$ and non-negative exponents $i$, $j$, $k$ and $l$, then we have
$$\lex(n,F)=O(n).$$
\end{corollary}

\begin{proof}
We apply Theorem~\ref{DSbound} for both bounds. The first bound follows because the stated upper bound holds for $\ds(n,w)$ for any word $w$, see \cite{K1992}.

The second bound follows from the fact if $w(F)$ has the form claimed, then $w'(F)$ must also have this form (with different exponents) and by the paper \cite{AKV1992} $\ds(n,w)$ is linear for such words $w$.

Note that Theorem~\ref{DSbound} does not apply if $F$ is a single star, but in this case an edge-ordered graph avoids $F$ if and only if its maximal degree is below the number $m$ of edges in $F$, so we have $\lex(n,F)=\lfloor(m-1)n/2\rfloor=O(n)$.
\end{proof}

Note that the Tur\'an number of a graph with at least two edges -- even without an edge-ordering -- is at least $\lfloor n/2\rfloor$. So the linear upper bound in Corollary~\ref{linear} is tight. It applies to every star forest with two star components and at most four edges. We finish the section by showing that a linear upper bound does not hold for a certain edge-ordering of the star forest consisting of a 2-edge star and a 3-edge star. The result is closely connected to the celebrated result of Hart and Sharir \cite{HS1986}  that we can state as $\ds(n,ababa)=\Theta(n\alpha(n))$. It is simpler for us, however, to derive our lower bound from a related result of F\"uredi and Hajnal \cite{FH1992}.

\begin{theorem}\label{DSlower}
Let $F$ be the edge-ordered star forest, with five edges such that the first, third and fifth edges form a star component and the second and fourth edges form another component. We have
$$\lex(n,F)=\Omega(n\alpha(n)),$$
where $\alpha(n)$ is the inverse Ackermann function.
\end{theorem}

\begin{proof}
F\"uredi and Hajnal proved in Corollary~7.5 of \cite{FH1992} that there exists an $n$ by $n$ 0-1 matrix $A_n$ with $\Theta(n\alpha(n))$ 1-entries that does not contain a submatrix of the form
$\begin{bmatrix}
     & 1 &  & 1 \\
    1 &  & 1 & 
\end{bmatrix}$, where the positions left blank could be arbitrary.

We build a bipartite graph $G_n$ such that $A_n$ is its adjacency matrix. $G_n$ has $2n$ vertices, $n$ of them (the row vertices) corresponding to the rows of $A_n$, and another $n$ (the column vertices) corresponding to the columns. The edges of $G_n$ correspond to the $1$ entries in $A_n$, so $G_n$ has $\Theta(n\alpha(n))$ edges.

We order the edges of $G_n$ left to right according to the column where the corresponding 1 entry appears. More precisely, an edge $e$ is less than another edge $e'$ if the 1 entry corresponding to $e$ is in a column that is to the left of the column containing the 1 entry corresponding to $e'$. We order the edges within the same column arbitrarily. We claim that the edge-ordered graph so obtained does not contain $F$. Assume for a contradiction that it contains $F$, so a subgraph of $G_n$ (as an edge-ordered graph) is isomorphic to $F$. We denote the vertices of $F$ by $a,b,c_1,c_2,c_3,c_4$ and $c_5$ as depicted in Figure \ref{fig1}.
\begin{figure}[h]
\begin{center}
\includegraphics[scale = 0.60] {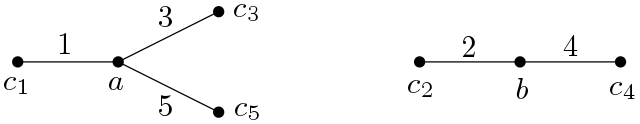}
\end{center}
\caption{The edge-ordered star forest $F$}
\label{fig1}
\end{figure}

We denote the corresponding vertices in the subgraph of $G_n$ by the corresponding upper case letters $A,B,C_1,C_2,C_3,C_3,C_4$ and $C_5$. Notice that the column vertices of $G_n$ are close, but neither central vertex $a$ or $b$ of $F$ is close, therefore $A$ and $B$ must be row vertices. The vertices $C_i$ are adjacent to $A$ or $B$, so they are column vertices. As the isomorphism preserves the edge-ordering, these columns $C_i$ must appear left to right in order of increasing indices. Rows $A$ and $B$ can be in either order. If row $A$ is below row $B$, then consider the 2 by 4 submatrix of $A_n$ formed by the rows $A$ and $B$ and the columns $C_1,\dots,C_4$. It is easy to see that this submatrix has a 1 entry in the four specified positions, contradicting the defining property of $A_n$. In case row $A$ is above row $B$, a similar contradiction comes from the 2 by 4 submatrix of $A_n$ formed by the rows $A$ and $B$ and the columns $C_2,\dots,C_5$.

The contradiction proves our claim that $G_n$ does not contain $F$ and thus shows that $\lex(2n,F)$ is at least the number of edges in $G_n$, so $\lex(2n,F)=\Omega(n\alpha(n))$. Using the monotonicity this implies the stated lower bound on $\lex(n,F)$.
\end{proof}

\section{Paths}
\label{sec:paths}

Let us start with introducing avoidance in an asymmetric bipartite context. It will play an important role in several of our results in this section.

By \emph{edge-ordered bipartite graphs} we mean an edge-ordered graph whose underlying graph is bipartite with a specified bipartition to \emph{left vertices} and \emph{right vertices}. If the edge-ordered graph $H$ has a specified root $x\in V(H)$, then we can distinguish if an edge-ordered bipartite graph $G$ contains $H$ with the root of $H$ being a left vertex or a right vertex. Accordingly, we say that $G$ \emph{left-contains} $H$ if a subgraph of $G$ is isomorphic to $H$ and the vertex corresponding to the root of $H$ is a left vertex in $G$. Otherwise we say, $G$ \emph{left-avoids} $H$. Similarly, we say $G$ \emph{right-contains} (or \emph{right-avoids}) $H$, according to whether $G$ has a subgraph isomorphic to $H$ in which the vertex corresponding to the root of $H$ is a right vertex. For this definition we consider the starting vertex of an edge labeled paths $P_k^L$ to be its root. Note that this definition depends on the presentation of $P_k^L$, for example $P_4^{132}$ and $P_4^{231}$ are isomorphic, but have different roots, so left-avoiding $P_4^{132}$ is the same as right-avoiding $P_4^{231}$.

As we have mentioned in the introduction, known results on the altitude of graphs imply a linear upper bound on the number of edges if a monotone labeling of the path $P_k$ is forbidden. First we prove a similar statement for any trees. We will use this to prove Lemma~\ref{majdnemmon}, which gives useful upper bound for the Tur\'an numbers of several edge-ordered paths.

We say that a labeling of a rooted tree $T$ is \emph{decreasing}, if the labels are decreasing on every branch (that is, on every path starting at the root). We call the labeling \emph{increasing} if the labels are increasing along every branch.

Note that in the next lemma we forbid all increasing (or all decreasing) labelings of a tree, rather than a specific one.

\begin{lemma}\label{montree} Let $T$ be a rooted tree of height $h$ with $t$ vertices. If an edge-ordered graph on $n$ vertices does not contain \emph{any} decreasing labeling of  $T$, then it has fewer than $htn$ edges. Moreover, if an edge-ordered bipartite graph $G$ does not left-contain any decreasing labeling of $T$, then it also has fewer than $htn$ edges.

The same bounds hold for edge-ordered graphs avoiding (or edge-ordered bipartite graphs left-avoiding) all increasing labelings of $T$.
\end{lemma}

\begin{proof}
By symmetry, it is enough to deal with graphs avoiding the decreasing labelings of $T$. Let $G$ be an edge-ordered graph  (or edge-ordered bipartite graph, respectively) with $n$ vertices and $htn$ edges. We will prove that $G$ contains (left-contains, respectively) a decreasing labeling of $T$ by induction on $h$. In case $h=1$, the average degree is larger than $t$, so $G$ contains the star $T$ with some labeling, but every labeling is decreasing. In the bipartite case the average degree of left vertices is larger than $t$, so $G$ left-contains $T$ as well.

For $h>1$ we delete the $t$ edges with smallest labels incident to every vertex $v$ of $G$ and let $G'$ be the resulting edge-ordered graph. In case a vertex of $G$ has degree less than $t$ we delete all incident edges. Let us delete the last level from $T$ (the vertices farthest from the root $x$) and let $T'$ be the resulting tree of height $h-1$.

$T'$ has fewer than $t$ vertices and $G'$ has at least $htn-tn=(h-1)tn$ edges. By induction, we can find a decreasing copy of $T'$ in $G'$ (with the root being a left vertex in the bipartite case). We extend $T'$ in $G$ greedily, adding edges one by one from the $t$ smallest edges from the given vertex (the ones we deleted from $G$). We do this till we obtain a copy of $T$. Whenever we add the edge, we have to make sure it avoids all the other vertices of the tree, at most $t-1$ vertices. This is doable, as there are $t$ edges to choose from. The monotonicity property is automatically satisfied by the way we chose the edges to delete from $G$.
\end{proof}

\begin{remark}
A more careful analysis of the proof gives that if we have $t_i$ vertices on level $i$, then the upper bound on the number of edges in $G$ can be improved to $n\sum_i (h-i+1)t_i$.
\end{remark}

Using the above lemma, we give a weaker bound for a couple specific orderings of paths. 
We call a labeling of a path $P$ \emph{monotone} if it is increasing or decreasing when considered with a root at one of the degree $1$ vertices.

\begin{lemma}\label{majdnemmon} Let $P$ be an edge-ordered path with a vertex $v$ that cuts it into two monotone paths $P'$ and $P''$, such that all labels of $P'$ are smaller than all labels of $P''$.
Then $\lex(n,P_k^L)=O(n\log n)$. 
\end{lemma}

\begin{proof}
Let us set $c=4k^3$, where $k$ is the number of vertices in $P$. We use induction on $n$ to prove that any edge-ordered graph $G$ with $n$ vertices and more than $cn\log n$ edges contains $P$.

Assume that our statement holds for smaller values of $n$ and let $G$ be an edge-ordered graph on $n$ vertices and more than $cn\log n$ edges. Our goal is to show that $G$ contains $P$. Let $G_1$ be the subgraph of $G$ formed by the set of the $\lceil\frac{c}{2}n\log n\rceil$ smallest edges of $G$ and let $G_2$ be the subgraph of $G$ formed by the remaining edges.

We consider both $P'$ and $P''$ as rooted trees with root $v$. Let $T$ be the rooted tree obtained by identifying the roots of $k$ pairwise disjoint copies of the path underlying of $P'$. We call a labeling of $T$ \emph{appropriate} if it is a decreasing labeling and the labeling of $P'$ is also decreasing or if it is an increasing labeling and the labeling of $P'$ is also increasing.

Let $V_1$ be the set of vertices that are roots of appropriately labeled copies of $T$ in $G_1$. We designate them as right vertices and the rest of the vertices as left vertices. Observe that there are at most $2k^3n$ edges of $G_1$ that are incident to a left vertex. Indeed, the subgraph of $G_1$ induced by the left vertices avoids all appropriate labelings of $T$, so it has at most $k^3n$ edges by Lemma~\ref{montree}, while the edge-ordered bipartite graph formed by the edges of $G_1$ between left and right vertices left-avoids all appropriate labelings of $T$, so  it has also at most $k^3n$ edges by the same lemma.

This implies that the subgraph of $G_1$ induced by $V_1$ has at least $\frac c2n\log n-2k^3n$ edges. It avoids $P$, so by induction it has at most $c|V_1|\log|V_1|$ edges. Therefore, we must have $|V_1|>n/2$.

Let $V_2$ be the set of vertices that are roots of an isomorphic copy of $P''$ in $G_2$. A similar argument shows that we must have $|V_2|>n/2$. This implies there is a vertex $x\in V_1\cap V_2$. Consider an isomorphic copy $P^*$ of $P''$ in $G_2$ rooted at $x$. Also, consider an appropriately labeled copy $T^*$ of $T$ in $G_1$ rooted at $x$. $T^*$ has $k$ branches, at least one of them does not meet $P^*$ outside the common root. Clearly, the union of this branch with $P^*$ is an isomorphic copy of $P$ in $G$.
\end{proof}

\begin{remark}
Let $T$ be an edge-ordered tree with a single vertex $v$ of degree larger than $2$. We call the maximal paths starting at $v$ the branches of $T$. A similar proof shows that if the branches are monotone and the edges of the branches form intervals in the edge-ordering, then $\lex(n,T)=O(n\log n)$.
\end{remark}

\subsection{Edge-ordered paths with three edges}\label{sec:threeedgepaths}

The path $P_4$ has three non-isomorphic labelings: $P_4^{123}$, $P_4^{132}$ and $P_4^{213}$.
This section is about their Tur\'an numbers. We determine  $\lex(n, P_4^{132})$ and $\lex(n,P_4^{213})$ exactly and $\lex(n, P_4^{123})$ up to an additive constant. %\tg{Not quite. How about $\lex(4k+2,P_4^{123})$? Do we want to include this or should we tone down this statement instead?}
First we prove a simple graph theoretical lemma that will be used for the proof of both results.

\begin{lemma}\label{no4ormore}
Let $G$ be a simple graph with $n\ge1$ vertices and $m$ edges that does not contain a cycle of length 4 or more. Then $m\le\frac{3}{2}(n-1)$.
\end{lemma}

\begin{proof}
We use induction by $n$. If there is no triangle in $G$, then $G$ is a forest and therefore $m\le n-1$ and we are done. Otherwise, it has a triangle $abc$. Let $G'$ be the graph obtained by removing the edges of this triangle from $G$. The vertices $a$, $b$ and $c$ fall in distinct components of $G'$ as any path connecting them in $G'$ could be extended by two edges of the triangle to a cycle of length at least four in $G$. Let $G_a$ and $G_b$ denote the connected component of the vertices $a$ and $b$ in $G'$, respectively, and let $G_c$ be the subgraph of $G'$ formed the remaining components. By the inductive hypothesis on these graphs we have
$$m=|E(G_a)|+|E(G_b)|+|E(G_c)|+3\le \frac{3}{2}(|V(G_a)|-1)+\frac{3}{2}(|V(G_b)|-1)+\frac{3}{2}(|V(G_c)|-1)+3=\frac{3}{2}(n-1).$$
\end{proof}

\begin{theorem}
\label{132thm}
$\lex(n, P_4^{132})=\lex(n,P_4^{213})=\left\lfloor\frac{3}{2}(n-1)\right\rfloor.$
\end{theorem}

\begin{proof} By symmetry (reversing the edge-order) it is enough to deal with $P_4^{132}$.

Consider any labeling of a cycle of length at least four. The subgraph formed by the largest edge in the cycle and its two adjacent edges is isomorphic to $P_4^{132}$. Thus, if an edge-ordered graph avoids $P_4^{132}$, then its underlying simple graph has no cycle of length at least four. The upper bound follows from Lemma~\ref{no4ormore}.

Now we will show that for every $n$, there is an edge labeled graph $G$ with $\left\lfloor\frac{3}{2}(n-1)\right\rfloor$ edges that avoids $P_4^{132}$. Let us obtain $G$ from an $n$-vertex star by adding to it a matching of size $\left\lfloor\frac{n-1}{2}\right\rfloor$ connecting leaves of the star. We have $|E(G)|=n-1+\left\lfloor\frac{n-1}{2}\right\rfloor=\left\lfloor\frac{3}{2}(n-1)\right\rfloor$ as needed. Label the edges in such a way that the edges of the original star receive the smallest labels. It is easy to check that the middle edge of any 3-edge path in $G$ is from the original star but the path has to also contain an edge outside this star. Therefore, $G$ avoids $P_4^{132}$.
\end{proof}

\begin{theorem}\label{123thm}
We have $\lex(n, P_4^{123}) \le \frac{3n}{2}$, with equality if and only if $n$ is divisible by $4$.
\end{theorem}

\begin{proof}
We start by describing two classes of graphs that have a monotone path of length 3 in any labeling. Odd cycles of length 5 or more are like that. Indeed, going around the cycle we can note if the label increases or decreases going from one edge to the next. This two cannot alternate because the cycle is odd, so we have two consecutive increases or two consecutive decreases. The three edges involved form a monotone path.

Now assume that a graph $G$ has four vertices $A, B_1, B_2$ and $B_3$ such that $A$ is connected to all three of $\{B_1, B_2, B_3\}$ and all three of $\{B_1, B_2, B_3\}$ has a neighbor not in $\{A, B_1, B_2, B_3\}$. (These neighbors may or may not coincide.) Any labeling $G^L$ contains a 3-edge monotone path. Indeed, we can assume by symmetry that $L(AB_1)<L(AB_2)<L(AB_3)$. Let $C$ be a neighbor of $B_2$ with $C\notin\{A,B_1,B_3\}$. If $L(B_2C)<L(AB_2)$, then $B_3AB_2C$ is a monotone path, otherwise $B_1AB_2C$ is.

We use induction on $n$ to prove the upper bound. The statement is trivial for $n\le 4$. Now assume that $n\ge 5$. Let $G$ be a graph with $n$ vertices and $m$ labeled edges with no monotone path of length 3.

If there is no cycle of length at least 4 in $G$, then Lemma \ref{no4ormore} implies $m\le\frac{3}{2}(n-1)<\frac{3n}{2}$. So $G$ contains a cycle of length at least 4. Let $C$ be a such a cycle of minimal length $t$. By our first observation, $t$ cannot be odd, so it is even.

First, assume that there is a vertex $A\in C$ connected to some vertex $B_1\not\in C$. Let $B_2$ and $B_3$ be the neighbors of $A$ in $C$. Then $B_1$ cannot be connected to $B_2$ or $B_3$, since that would create an odd cycle of length $t+1$. $B_1$ cannot be connected to a vertex not in $\{A,B_2,B_3\}$ either, since this would create the other type of forbidden subgraph we described before. Therefore the only neighbor of $B_1$ is $A$. Let us delete $B_1$ from $G$. By induction the remaining graph has at most $\frac{3}{2}(n-1)$ edges, therefore $m\le \frac{3}{2}(n-1)+1<\frac{3n}{2}$.

Now assume that there is no edge connecting a vertex of $C$ to a vertex not in $C$, that is the vertices of $C$ form a component of $G$. The rest of the graph contains at most $\frac32(n-t)$ edges by induction. If $t>4$, then $C$ must be an induced cycle as a chord in $C$ would create a shorter cycle still of length at least 4, so we have $m\le t+\frac32(m-t)<\frac32m$. Finally if $t=4$, then the component of $C$ can contain at most $6$ edges and we have $m\le6+\frac{3}{2}(n-4)\le\frac32m$. We can only have equality in this case and (by induction) only if all components of $G$ are cliques of size 4.

To completely characterize the cases of equality in the theorem, it is enough to show that a disjoint union of copies of $K_4$ can be labeled in a way avoiding $P_4^{123}$. Clearly, it is enough to label one component. A labeling of $K_4$ avoids $P_4^{123}$ if and only if both the two smallest and the two largest labels are given to pairs of independent edges.
\end{proof}

\begin{corollary}
$6\lfloor \frac{n}{4}\rfloor=\lex(4\lfloor \frac{n}{4}\rfloor, P_4^{123}) \le \lex(n, P_4^{123}) \le \frac{3n}{2}$, which determines $\lex(n, P_4^{123})$ up to an additive constant.
\end{corollary}

\subsection{Edge-ordered paths with four edges}
\label{sec:fouredge}

The labelings (or edge-orderings) of $P_5$ are given by permutations of $\{1,2,3,4\}$. However, two reverse permutations (e.g., $1324$ and $4231$) yield isomorphic labeled graphs. Also, the Tur\'an number remains the same if we reverse the edge-ordering. For example if $G$ is a $P_5^{1243}$-free labeled graph, then reversing the edge-ordering in $G$ gives a $P_5^{4312}$-free graph. Therefore, the Tur\'an numbers of the two or four labelings are equal in each of the eight classes in the following table. For each of these equivalence classes, we summarize the upper and lower bound we prove on $\lex(n,P_5^L)$.

\vspace{2mm}
\begin{table}[h!]
\begin{center}
\begin{tabular}{ |p{5cm}||p{3cm}|p{3cm}|p{3.5cm}|  }
 \hline
 \multicolumn{4}{|c|}{Tur\'an numbers of edge-ordered paths with four edges} \\
 \hline
 Labeling & Lower bound & Upper bound & Proved in\\
 \hline
$\{1234,4321\}$  & $\Omega(n)$    & $O(n)$ & Prop.~\ref{thm1234} (i) \\
$\{1243,3421,4312,2134\}$ & $\Omega(n)$    & $O(n)$ & Prop.~\ref{thm1234} (ii)\\
$\{1324,4231\}$ & $\Omega(n \log n)$ & $O(n\log n)$ & Thm.~\ref{thm1324}\\
$\{1432,2341,4123,3214\}$ & $\Omega(n\log n)$  & $O(n\log n)$  & Thm.~\ref{thm1432}\\
$\{2143,3412\}$ &$\Omega(n\log n)$  & $O(n\log n)$ & Thm.~\ref{recur} (ii)\\
$\{1342,2431,4213,3124\}$   &$\Omega(n\log n)$ & $O(n\log^2 n)$ & Thms.~\ref{recur} (i), \ref{thm1342up}\\
$\{2413,3142\}$ & $\binom{n}{2}$ & $\binom{n}{2}$ & Prop. \ref{thm1423} (ii)\\
$\{1423,3241,4132,2314\}$ &  $\binom{n}{2}$ & $\binom{n}{2}$ & Prop. \ref{thm1423} (i) \\
\hline
\end{tabular}
%\caption{}
\end{center}
\end{table}
\vspace{2mm}

\begin{proposition}\label{thm1234} We have

\smallskip 

\textbf{(i)} $\lex(n,P_5^{1234})=\Theta(n)$, and

\smallskip 

\textbf{(ii)} $\lex(n,P_5^{1243})=\Theta(n)$.

\end{proposition}

\begin{proof} The lower bounds are obvious in both cases. We mentioned earlier the linear upper bound for monotone paths of any length. That implies the upper bound in \textbf{(i)} but we prove it together with \textbf{(ii)} to obtain the same upper bound of $9n/2$ that is stronger than what follows from the earlier proof.

Let us consider an edge-ordered graph $G$ on $n$ vertices with more than $9n/2$ edges. Our goal is to prove that $G$ contains both $P_5^{1243}$ and $P_5^{1234}$. For every vertex $v$ of $G$, we remove the smallest three edges incident to $v$ (or all incident edges if the degree of $v$ is less than $3$). This way we remove at most $3n$ edges, thus the resulting graph $G'$ has more that $3n/2$ edges.

By Theorem~\ref{132thm}, $G'$ contains $P_4^{132}$. Let $v_1,v_2,v_3,v_4$ be the vertices of a subgraph of $G'$ isomorphic to $P_4^{132}$, so $v_1v_2$, $v_3v_4$ and $v_2v_3$ are edges of $G'$ ordered in this order. Recall that we removed the three smallest edges incident to $v_1$ from $G$. The other endpoint of at least one of these three edges is different from $v_3$ and $v_4$. Choosing such a vertex $u$ as a starting vertex, we obtain the path $uv_1v_2v_3v_4$ in $G$, and its labeling makes it isomorphic to $P_5^{1243}$.

Observe that $G'$ also contains a $P_4^{123}$ by Theorem~\ref{123thm}, and then the same reasoning as above yields that $G$ also contains $P_5^{1234}$.
\end{proof}

Note that by Theorem~\ref{ess} the next statement is equivalent to $\chi_<'(P_5^{1423})=\chi_<'(P_5^{2413})=\infty$.

\begin{proposition} \label{thm1423} We have

\smallskip 

\textbf{(i)} $\lex(n,P_5^{1423})=\binom{n}{2}$, and 

\smallskip 

\textbf{(ii)} $\lex(n,P_5^{2413})=\binom{n}{2}$.
\end{proposition}

\begin{proof}  This follows directly from the fact that the max-labeling of $K_n$ avoids both $P_5^{1423}$ and $P_5^{2413}$. This last statement follows from Proposition~\ref{close} as neither end vertex of the largest edge is close in either of the edge-ordered paths $P_5^{1423}$ and $P_5^{2413}$.
\end{proof}

We prove several of the lower bounds of the form $\lex(n,P)=\Omega(n\log n)$ by constructing edge-ordered graphs $G_i$ avoiding $P$ such that $G_i$ has $2^i$ vertices and $\Omega(i2^i)$ edges. This is enough by the monotonicity of $\ex(n,P)$. Indeed, if $P$ has no isolated vertices than one can add isolated vertices to any edge-ordered graph avoiding $P$ to obtain an edge-ordered graph on more vertices and the same number of edges, still avoiding $P$. So in the situation above we have $\lex(n,P)\ge\lex(2^{\lfloor\log n\rfloor},P)=\Omega(n\log n)$.

\begin{theorem}\label{recur} We have 

\smallskip 

\textbf{(i)} $\lex(n,P_5^{1342})=\Omega(n\log n)$, and

\smallskip

\textbf{(ii)} $\lex(n,P_5^{2143})=\Theta (n\log n)$.
\end{theorem}

\begin{proof} The upper bound of \textbf{(ii)} follows from Lemma \ref{majdnemmon}.

To prove \textbf{(i)}, we build the $P_5^{1342}$-free edge-ordered graphs $G_i$ recursively. Let $G_0$ be a single vertex. To construct $G_{i+1}$ we take two copies of $G_i$ and add a perfect matching $M$ between the two copies. Note that we can take an arbitrary perfect matching. 
We keep the order of the edges within both copies of of $G_i$, but make all edges in one copy (the large copy) larger than any edge in the other copy (the small copy). We further make all edges in the matching $M$ larger than any other edge. The order among the matching edges is arbitrary.

Clearly, $G_i$ has $2^i$ vertices and $i2^{i-1}$ edges. It remains to prove it avoids $P_5^{1342}$. We do this by induction on $i$. The statement trivially holds for $G_0$, so assume it holds for $G_i$ and assume for a contradiction that an isomorphic copy $P_5^{1342}$ shows up in $G_{i+1}$ formed by edges $e_1$, $e_2$, $e_3$ and $e_4$ with the edge-ordering $e_1<e_4<e_2<e_3$. It cannot be completely inside a copy of $G_i$ by the inductive hypothesis, but it is connected, so it has to contain an edge from $M$. As $e_3$ is the largest of the four edges it must come from $M$ and therefore $e_2$ and $e_4$ (being incident to distinct end points of $e_3$) must come from the two separate copies of $G_i$, $e_2$ coming from the large copy and $e_4$ from the small copy. The edge $e_1$ is adjacent to $e_2$ from the large copy but smaller than $e_4$ from the small copy, a contradiction.

The lower bound of \textbf{(ii)} is given by a similar recursive construction. We construct the $P_5^{2143}$-free edge-ordered graphs $G'_i$ similarly. $G'_0$ is a single vertex, and $G'_{i+1}$ is obtained by connecting two disjoint copies of $G'_i$ by a perfect matching $M$, but the edge-ordering is different. We still keep the edge-orderings inside both copies of $G'_i$ and make all edges of one copy larger than any edge of the other copy, but this time the edges of $M$ will be intermediate: larger than the edges in the small copy of $G'_i$ and smaller than the edges in the large copy. We can choose the perfect matching $M$ arbitrarily and the order of the edges inside $M$ is arbitrary too.

We still have that $G'_i$ has $2^i$ vertices and $i2^{i-1}$ edges. For the inductive proof that these graphs avoid $P_5^{2143}$ assume for a contradiction that $G'_i$ avoids it but $G'_{i+1}$ has an isomorphic copy formed by the edges $e_1$, $e_2$ $e_3$ and $e_4$ with the edge-ordering $e_2<e_1<e_4<e_3$. Here $G'_{i+1}$ consist of two copies of $G'_i$ connected by a matching $M$. If $e_2\in M$, then two adjacent edges $e_1$ and $e_3$ are in different copies of $G'_i$, thus one of them should be smaller than $e_1$, a contradiction. Similarly, if $e_3\in M$, then one of its adjacent edges $e_2$ or $e_4$ should be larger than $e_4$, a contradiction. So $e_2$ and $e_3$ are in one of the copies of $G'_i$ and as they are adjacent, it is the same copy. As the other two edges are in between them in the ordering they should also be in the same copy of $G'_i$ contradicting the inductive assumption that $G'_i$ avoids $P_5^{2143}$.
\end{proof}

\begin{theorem}\label{thm1432}
$\lex(n,P_5^{1432})=\Theta (n\log n)$.
\end{theorem}

\begin{proof} The upper bound follows from Lemma \ref{majdnemmon}. It can also be derived from Lemma~4.1(c) in \cite{T2013}. To prove the lower bound we use a result of F\"uredi \cite{F1990}: there exist $n \times n$ 0-1 matrices $A_n$ that do not contain the submatrix 
$\begin{bmatrix} 
 & 1 & 1 \\
 1 &  & 1
\end{bmatrix}$ (with arbitrary entries in the two places left blank) and have $\Omega(n \log n)$ $1$ entries.

We build a bipartite graph $G_n$ such that $A_n$ is its adjacency matrix. $G_n$ has $2n$ vertices, $n$ of them (the row vertices) corresponding to the rows of $A_n$, and another $n$ (the column vertices) corresponding to the columns. The edges of $G_n$ correspond to the $1$ entries in $A_n$, so $G_n$ has $\Omega(n \log n)$ edges. Now we make $G_n$ into an edge-ordered graph by ordering its edges according to the corresponding $1$ entries in $A_n$. If two $1$ entries are in distinct columns, the one to the right is larger. If they are in the same column, then the one lower in the column is larger.

It remains to prove $G_n$ avoids $P_5^{1432}$. Assume for a contradiction that an isomorphic copy of $P_5^{1432}$ shows up in $G_n$. Notice that column vertices are close, but the common vertex of the first two edges in $P_5^{1432}$ is not close, so it must correspond to a row vertex. This means that two rows and three columns corresponding to the five vertices of the copy of $P_5^{1432}$ form exactly the the forbidden type of submatrix, a contradiction.
\end{proof}

Our next lemma will be used in proving Theorem~\ref{thm1324}.

\begin{lemma}\label{rightright}
The maximal number of an edge-ordered bipartite graph on $n$ vertices that right-avoids both $P_4^{132}$ and $P_4^{213}$ is $\Theta(n\log n)$.
\end{lemma}

\begin{proof} The proof of the lower bound is similar to the recursive construction of the lower bounds in Theorem \ref{recur}. We build the edge-ordered bipartite graphs $G_i$ recursively. The graph $G_i$ will have $2^{i-1}$ left vertices, $2^{i-1}$ right vertices and $(i+1)2^{i-2}$ edges.

We start with $G_1$ being (the only edge-ordering of) $K_2$ and we designate one of the vertices left, the other right. For $i\ge1$ we construct $G_{i+1}$ as follows. We take two disjoint copies of $G_i$ and connect them by a perfect matching $M$ between the left vertices of the first copy and the right vertices of the second copy. We keep the order of the edges within either of the two copies of $G_i$ and make the edges of $M$ larger than the edges in the first copy of $G_i$ and smaller than the edges in the second copy of $G_i$. Note that the matching $M$ and the order of edges within $M$ is arbitrary. A left vertex of either copy of $G_i$ will also be a left vertex of $G_{i+1}$, while a right vertex of either copy of $G_i$ is also a right vertex of $G_{i+1}$.

We claim that $G_i$ right-avoids $P_4^{132}$ for all $i$. We prove this by induction on $i$. It trivially holds for $G_1$. Assume that $G_i$ right-avoids $P_4^{132}$ but we still find a copy $P$ of $P_4^{132}$ in $G_{i+1}$ formed by the edges $e_1$, $e_2$ and $e_3$ ordered as $e_1<e_3<e_2$. We need to prove it starts at a left vertex. If $P$ is contained in one of the copies of $G_i$, then it starts as a left vertex by the inductive hypothesis. Otherwise one of the edges in $P$ is in the matching $M$. Notice that $M$ is an induced matching (no edge of $G_{i+1}$ connects two edges in $M$), so exactly one edge of $P$ is in $M$. It cannot be $e_2$ as one of the other two edges of $P$ would then be in the second copy of $G_i$ and would be larger than $e_2$. The matching edge cannot be $e_3$ either as then $e_1$ and $e_2$ (being smaller and larger than $e_3$, respectively) would be in different copies of $G_i$ and could not be adjacent. So $e_1$ must be in $M$, and then then $e_2$ and $e_3$ (being larger than $e_1$) are in the second copy of $G_i$, so $P$ starts in the left side of the first copy of $G_i$ proving the claim.

A similar inductive proof shows that $G_i$ right-avoids $P_4^{213}$ for all $i$. Indeed if $G_i$ right-avoids $P_4^{213}$ but an isomorphic copy $P$ of $P_4^{213}$ consisting of the edges $e_1$, $e_2$, and $e_3$ in the edge-ordering $e_2<e_1<e_3$ shows up in $G_{i+1}$, then either $P$ is contained in a single copy of $G_i$ and then it starts at a left vertex or exactly one of its edges is in the matching $M$. As above, the edge in $M$ cannot be $e_1$ or $e_2$, so it must be $e_3$, but then $P$ starts at a left vertex again. This finishes the proof of the lower bound.

For the upper bound we consider an edge labeled bipartite graph $G^L$ on $n$ vertices that right-avoids both $P_4^{132}$ and $P_4^{213}$. That is, $G$ is a bipartite graph with a given bipartition to left and right vertices and $L$ is an injective labeling $L:E(G)\to\reals$ specifying the edge-ordering. As only the relative order of the labels matters we can assume without loss of generality that the labels are integers between $1$ and $|E(G)|$. This assumption is needed because in the proof below we compare not only the labels but also distances between labels.

For a non-isolated left vertex $x$ in $G$ we write $m(x)$ for the minimal label $L(e)$ of an edge $e$ incident to $x$. For a non-isolated right vertex $y$ in $G$ we write $M(y)$ for the maximal label $L(e)$ for an edge $e$ incident to $y$. For an edge $e=xy$ of $G$ with $x$ a left vertex and $y$ a right vertex we call $L(e)-m(x)$ the \emph{left-weight} of $e$ and $M(y)-L(e)$ is the \emph{right-weight} of $e$. Note that these are non-negative integers less than $|E(G)|<n^2$. We call the edge \emph{left-leaning} if its left-weight is larger than its right-weight, otherwise it is \emph{right-leaning}.

Let $xy_1$ and $xy_2$ be two edges of $G$ incident to the same left vertex $x$ and assume $L(xy_2)>L(xy_1)$. Then the left weight of $xy_2$ is larger than the left-weight of $xy_1$. We must further have $M(y_1)<L(xy_2)$ as otherwise the path $y_2xy_1$ followed by the edge of label $M(y_1)$ would show that $G$ right-contains $P_4^{213}$, contrary to our assumption. If we further assume that $xy_1$ is right-leaning, then we have $L(xy_2)-m(x)>M(y_1)-m(x)\ge2(L(xy_1)-m(x))$, so the left-weight of $xy_2$ is more than twice of that of $xy_1$. As all these left-weights are non-negative integers below $n^2$, this implies that there are $O(\log n)$ right-leaning edges of $G$ incident to $x$. The total number of right-leaning edges is therefore $O(n\log n)$.

To obtain a similar bound for left-leaning edges, let $x_1y$ and $x_2y$ be two edges of $G$ incident to the same right vertex $y$ with $L(x_1y)<L(x_2y)$. We have $m(x_2)>L(x_1y)$ as otherwise the path starting with the edge of label $m(x_2)$ continued by $x_2yx_1$ would show that $G$ right-contains $P_4^{132}$. So if $x_2y$ is left-leaning, then we have $M(y)-L(x_1y)>M(y)-m(x_2)\ge2(M(y)-L(x_2y))$, so the right-weight of $x_1y$ is more than twice the right-weight of $x_2y$. As before, this implies that the number of left-leaning edges incident to $y$ is $O(\log n)$ and the total number of left-leaning edges in $G$ is $O(n\log n)$. As every edge of $G$ is either left- or right-leaning, $G$ has $O(n\log n)$ edges. This finishes the proof of the upper bound.
\end{proof}

%\begin{figure}[h]
%\begin{center}
%\includegraphics[scale = 0.4] {edgeord_fig3}
%\end{center}
%\caption{}
%\label{fig3}
%end{figure}
%Removed this figure as it no longer matches the rewritten proof -ND.

\begin{theorem}\label{thm1324} $\lex(n,P_5^{1324})= \Theta(n\log n)$.
\end{theorem}

\begin{proof}
Consider an edge-ordered bipartite graph $G$ that right-contains $P_5^{1324}$. The first three edges of the isomorphic copy of $P_5^{1324}$ forms an isomorphic copy of $P_4^{132}$, so $G$ also right-contains $P_4^{132}$. Similarly, if $G$ left-contains $P_5^{1324}$, then the last three edges of the isomorphic copy of $P_5^{1324}$ is isomorphic to $P_4^{213}$, so $G$ right-contains $P_4^{213}$. Therefore, if an edge-ordered bipartite graph $G$ right-avoids both $P_4^{132}$ and $P_4^{213}$, then $G$ both right- and left-avoids $P_5^{1324}$, so it avoids $P_5^{1324}$. By Lemma~\ref{rightright} such edge-ordered bipartite graphs $G$ exist with $n$ vertices and $\Omega(n\log n)$ edges proving the lower bound in the theorem.

For the upper bound we will also use Lemma~\ref{rightright} but we need a more involved deduction. Let $G^L$ be an edge-ordered graph with $n$ vertices and $m$ edges avoiding $P_5^{1324}$. Our goal is to prove $m=O(n\log n)$.

First we partition the vertex set of $G$ to left and right vertices and consider the bipartite subgraph $G'$ of $G$ formed by the edges between the left and the right vertices. We can do the partition in such a way, that $G'$ contains at least $m/2$ edges. We remove the edge with the minimal label incident to every vertex, and obtain the subgraph $G''$ with the edge set $E$. We clearly have $|E|\ge m/2-n$. When saying that $E$ or a subset of $E$ avoids (or left- or right-avoids) a pattern we mean the statement for the edge-ordered bipartite subgraph of $G'$ formed by those edges. In particular, $E$ avoids both $P_5^{1324}$ and $C_4^{1324}$. It avoids the former because the entire edge-ordered graph $G$ avoids it. 
%It avoids the latter because in case a copy of $C_4^{1324}$ shows up in $E$, then we can replace the edge with the lowest label in the cycle with another edge of $G'$ outside $E$ to obtain a copy $P_5^{1324}$ in $G'$, a contradiction.
Assume  a copy of $C_4^{1324}$ shows up in $E$. Let $x$ be the vertex incident with the edges with the smallest and third smallest label in this cycle. We have deleted from $G'$ the edge $xy$ with the minimal label incident to $x$. As $G'$ is bipartite, $y$ is not on the four-cycle, thus we can replace the edge with the lowest label in the cycle with $xy$ to obtain a copy $P_5^{1324}$ in $G'$, a contradiction.

For a non-isolated vertex $x$ of $G''$ let $m(x)$ (respectively, $M(x)$) stand for the minimal (respectively, maximal) label $L(e)$ of an edge $e\in E$ incident to $x$. For an edge $xy\in E$ with $x$ a left vertex and $y$ a right vertex we write $S(xy)$ (respectively, $T(xy)$) for the set of edges $x'y\in E$ with $m(x)<L(x'y)<L(xy)$ (respectively, with $L(xy)<L(x'y)<M(x)$). Let $S=\{e\in E\mid|S(e)|\le1\}$ and $T=\{e\in E\mid|T(e)|\le1\}$.

Let us form an auxiliary graph with the vertex set $S$ by connecting $e\in S$ to the at most one element in $S(e)\cap S$. Now $e$ is connected to at most one other edge of label less than $L(e)$, so this auxiliary graph is a forest. Forests are bipartite, so we can partition $S$ into the independent sets $S_1$ and $S_2$. Note that if $y'xyx'$ is an isomorphic copy of $P_4^{132}$ starting at a right vertex $y'$, then $m(x)\le L(xy')<L(x'y)<L(xy)$, therefore $x'y\in S(xy)$. This means, that $S_1$ and $S_2$, being independent sets in the auxiliary graph cannot contain such a path, so both $S_1$ and $S_2$ right-avoids $P_4^{132}$.

Similarly, the auxiliary graph on the vertex set $T$, where $e\in T$ is connected to the at most one element of $T(e)\cap T$ is a forest, so $T$ can be partitioned into the independent sets $S_3$ and $S_4$. As above, both $S_3$ and $S_4$ left-avoid $P_4^{213}$ because the first two edges of any left-starting copy of $P_4^{213}$ in $T$ are connected in this auxiliary graph.

Let $y$ be a right vertex and let $xy$ and $x'y$ be edges in $E$ with $L(xy)<L(x'y)$. Extend the path $x'yx$ at $x'$ with the edge of label $m(x')$ and at $x$ with the edge of label $M(x)$. Unless $m(x')>L(xy)$ or $M(x)<L(x'y)$ we obtain a copy of $P_5^{1324}$ or (if the two edges added are adjacent) a copy of $C_4^{1324}$. As $E$ avoids both of these patterns we must have $m(x')>L(xy)$ or $M(x)<L(x'y)$. Assume now that there is at most one edge $x''y\in E$ with $L(xy)<L(x''y)<L(x'y)$. Then $m(x')>L(xy)$ implies $x'y\in S$ and $M(x)<L(x'y)$ implies $xy\in T$.

Arrange the edges in $E$ incident to the right vertex $y$ according to their labels. By the previous paragraph, if two edges are consecutive or second neighbors in this list, then one of them must be in $S\cup T$. As a consequence we have $|S\cup T|\ge2|E|/3-n$. Note that $S\cup T$ can be covered by four sets (namely $S_1$, $S_2$, $S_3$ and $S_4$), each of which either right-avoids $P_4^{132}$ or left-avoids $P_4^{213}$.

Notice that we kept the left-right symmetry when defining the edge set $E$, so the statements in the last paragraph have their mirror images too. In particular, there exists a set $U\subseteq E$ with $|U|\ge2|E|/3-n$ such that $U=U_1\cup U_2\cup U_3\cup U_4$ and each set $U_i$ either left-avoids $P_4^{132}$ or right-avoids $P_4^{213}$.

But now we have $|(S\cup T)\cap U|\ge|E|/3-2n$ and $(S\cup T)\cap U=\bigcup_{i,j}S_i\cap U_j$. Here 8 out of the 16 intersections satisfies that $S_i\cap U_j$ (left- and right-avoids and therefore) avoids either $P_4^{132}$ or $P_4^{213}$, in which case $|S_i\cap U_j|=O(n)$ by Theorem~\ref{132thm}, while in another 8 cases $S_i\cap U_j$ either right-avoids both of $P_4^{132}$ and $P_4^{213}$ or left-avoids both of them. We have $|S_i\cap U_j|=O(n\log n)$ in the right-avoiding case by Lemma~\ref{rightright} and the same bound holds by symmetry in the left-avoiding case.

Summarizing, we must have $|(S\cup T)\cap U|=O(n\log n)$ and therefore $|E|=O(n\log n)$ and finally we must also have $m=O(n\log n)$ for the number $m$ of edges in $G$. This finishes the proof of the upper bound.
\end{proof}

To prepare for our final result about four edge paths, namely Theorem~\ref{thm1342up}, we start with the following lemma. Note that while we do not expect Theorem~\ref{thm1342up} to be tight, this lemma is tight. Indeed, a slight modification of the construction given in the proof of Theorem~\ref{recur}(i) for the edge-ordered graphs $G_i$ avoiding $P_5^{1342}$ yields edge-ordered bipartite graphs avoiding $P_5^{1342}$ and also right-avoiding $P_4^{132}$. One only has to maintain a bipartition of the constructed graph $G_i$ to an equal number of left and right vertices and (when constructing $G_{i+1}$ from $G_i$) to restrict the matching to connect the right vertices in the smaller copy of $G_i$ to the left vertices in the larger copy.

Also note that this lemma follows directly from Lemma~4.1(b) in \cite{T2013}. We include the proof to be self contained.

\begin{lemma}\label{bipno132}
If an edge-ordered bipartite graph on $n$ vertices avoids $P_5^{1342}$ and right-avoids $P_4^{132}$, then it has $O(n\log n)$ edges. 
\end{lemma}

\begin{proof} Let $H^L$ be the bipartite edge labeled graph on $n$ vertices that avoids $P_5^{1342}$ and right-avoids $P_4^{132}$. Our goal is to bound the number of edges in $H$. As in the proof of the upper bound in Lemma~\ref{rightright} we will compare differences between labels of edges, and to make this meaningful we assume $L$ takes integer values between $1$ and $n^2$.

First we delete the smallest labeled edge incident to each non-isolated vertex of $H^L$ to obtain the subgraph $H'$. We lose less than $n$ edges and $H'$ avoids $C_4^{1342}$. Indeed, if a copy of $C_4^{1342}$ showed up in $H'$ we could replace the smallest edge in the cycle with one of the edges not in $H'$ to obtain a copy of $P_4^{1342}$ in $H^L$, a contradiction. (Note, we did exactly the same thing in the proof of Theorem~\ref{thm1324} to obtain a large  subgraph of a graph avoiding $P_5^{1324}$ that avoids $C_4^{1324}$.)

For each non-isolated right vertex $y$ in $H'$ we define $m(y)$ to be the smallest label of an edge of $H'$ incident to $y$. When referring to an edge $xy$ of $H'$ we will always assume $x$ is a left vertex and $y$ is a right vertex. With this notation we define the \emph{weight} of an edge $xy$ of $H'$ to be $w(xy)=L(xy)-m(y)$. We call the edge $xy$ of $H'$ \emph{minimal} if $w(xy)=0$, otherwise we define $n(xy)$ to be the label of the ``next smallest label at $y$'', that is $n(xy)=\max L(x'y)$, where the maximum is taken for edges $x'y$ in $H'$ with $L(x'y)<L(xy)$. For a non-minimal edge $xy$ of $H'$ we compare $L(xy)-n(xy)$ and $n(xy)-m(y)$. If the former is larger, then $xy$ is \emph{light}, otherwise it is \emph{heavy}. 

The weight of light edge is more than twice of the weight of any other edge of smaller weight incident to the same right vertex. Therefore, the number of light edges incident to any one right vertex is $O(\log n)$ and the total number of light edges in $H'$ is $O(n\log n)$. Clearly, the number of minimal edges in $H'$ is at most $n$, while the number of edges of $H$ not in $H'$ is also at most $n$. To finish the proof of the lemma it remains to limit the number of heavy edges.

Take two heavy edges from the same left vertex $x$: $xy$ and $xy'$ with $L(xy)<L(xy')$. Extend the 2-edge path $yxy'$ at $y'$ with the edge labeled $n(xy')$. As $H^L$ right-avoids $P_4^{132}$ we must have $n(xy')<L(xy)$. We further extend the 3-edge path at $y$ with the edge labeled $m(y)$. If $m(y)<n(xy')$, we obtain an isomorphic copy of $P_5^{1342}$ or $C_4^{1342}$. As $H'$ avoids both we must have $m(y)>n(xy')$. As $xy'$ is heavy we have $w(xy')\ge2(L(xy')-n(xy'))$. But $L(xy')-n(xy')>l(xy)-m(y)=w(xy)$. So the weight doubles from one heavy edge incident to a given left vertex $x$ to the next heavy edge. Therefore, the number of heavy edges incident to $x$ is $O(\log n)$ and the total number of heavy edges in $H'$ is $O(n\log n)$, proving the lemma.
\end{proof}

\begin{theorem}
\label{thm1342up}
$\lex(n,P_5^{1342})= O(n\log^2 n)$.
\end{theorem}

\begin{proof}
Let $G$ be an edge-ordered graph avoiding $P_5^{1342}$ on $n$ vertices with a maximal number of $m=\lex(n,P_5^{1342})$ edges. Let $L$ be the graph formed by the $\lfloor m/2\rfloor$ smallest edges (lower half) of $G$ and let $U$ be the subgraph formed by the remaining edges (upper half) of $G$. We call a vertex a \emph{left vertex} it has at least $3$ incident edges in $L$, otherwise it is a \emph{right vertex}. (Note that neither $L$ nor $U$ must be bipartite though.)

Clearly, $L$ has at most $3n$ edges not between two left vertices, so at least $m/2-3n$ edges between left vertices.

We claim that $U$ does not contain an isomorphic copy of $P_4^{132}$ that ends at a left vertex. Indeed, such a copy could be extended at its end with an edge from $L$. We have at least three choices for this last edge, so at least one yields a simple four edge path and that would be isomorphic to $P_5^{1342}$, a contradiction.

Let $U_{\mathrm{bip}}$ be the edge-ordered bipartite graph consisting of the edges in $U$ between a left and a right vertex. As a subgraph of $G$ it avoids $P_5^{1342}$ and as shown above it also right-avoids $P_4^{132}$, so by Lemma~\ref{bipno132}, $U_{\mathrm{bip}}$ contains $O(n \log n)$ edges.

The edges of $U$ between left vertices form an edge-ordered graph avoiding $P_4^{132}$ by the same claim above. Thus, by Theorem~\ref{132thm} $O(n)$ edges of $U$ connect two left vertices. 

By the previous two paragraphs, there are $m/2-O(n \log n)$ edges of $U$ connecting right vertices. They form an edge-ordered graph avoiding $P_5^{1342}$ just as the $m/2-O(n)$ edges of $L$ between left vertices do. We have either at most $\lfloor n/2\rfloor$ left vertices or at most at most $\lfloor n/2\rfloor$ right vertices, and either case we must have $ \lex(\lfloor n/2\rfloor,P_5^{1342})\ge m/2-O(n \log n)$. We can rewrite this as $\lex(n,P_5^{1342})\le2\lex(\lfloor n/2\rfloor,P_5^{1342})+O(n \log n)$. This recursion solves to $\lex(n,P_5^{1342}) = O(n \log^2 n)$, as claimed.
\end{proof}

\subsection{Longer paths}

Some of our results above directly imply bounds for longer edge-ordered paths as well.
For example, we have a linear upper bound for the Tur\'an numbers of monotone paths of any length and
Lemma~\ref{majdnemmon} gives an upper bound for the Tur\'an numbers of some other edge-orderings of longer paths. Some of our constructions in the previous section can be shown to avoid more edge-ordered paths than what was shown. See more on this in Section~\ref{gene}.

Our results in the previous sections imply that all edge-orderings of $P_4$ have order chromatic number $2$ and all edge-orderings of $P_5$ have order chromatic number $2$ or infinity. Here we show that this does not remain the case for the edge-orderings of $P_6$. By Theorem~\ref{ess}, our result below implies that $\lex(n,P_6^{14325})=n^2/4+o(n^2)$. This shows a very different asymptotic behaviour compared to the Tur\'an numbers for shorter edge-ordered paths.

%\tg{Do we know that all order chromatic numbers show up for longer paths? Do we want to ask this?}
%\dom{It's not a bad question, but this project has reached that stage that I would rather not start thinking about new stuff. Since we haven't seriously thought about it, I would not ask it, but I don't mind much if you do.}

\begin{theorem}
$\chi_<'(P_6^{14325})=3$.
\end{theorem}

\begin{proof}
The inequality $\chi_<'(P_6^{14325})\ge3$ follows directly from Proposition~\ref{close}. Indeed $P_6^{14325}$ has a single proper two-coloring and both color class contains a vertex that is not close.

For the reverse direction we use Theorem~\ref{canon}. It is enough to show that all canonical edge-orders of $K_{3\times3}$ contain $P_6^{14325}$. For this note that for all canonical edge-orders for $K_{3\times3}$ either one of the three parts precedes the other two parts or one of the three parts is preceded by the other two parts. (We have already used this fact in the proof of Proposition~\ref{non-principal}.) In the former case we can find an isomorphic copy of $P_5^{1432}$ in the minimal part and then we can extend it to $P_6^{14325}$ by an edge outside this part. In the latter case we find an isomorphic copy of $P_5^{4325}$ in the maximal part and extend it to $P_6^{14325}$ using an edge outside this part.
\end{proof}

\section{4-cycles}
\label{sec:Fourcycle}
The four edge cycle $C_4$ has three non-isomorphic edge-orderings. The only one which embeds into a max-labeled  clique is $C_4^{1243}$, therefore $\chi_<'(C_4^{1234})=\chi_<'(C_4^{1324})=\infty$ and
$$\lex(n, C_4^{1234})=\lex(n, C_4^{1324})=\binom{n}{2}.$$
In this section we improve the upper bound $\lex(n,C_4^{1243})=O(n^{5/3})$ proved in \cite{GPV2017}. Our proof is inspired by some ideas of \cite{MT2006}. Note the simple lower bound $\lex(n,C_4^{1243})\ge\ex(n,C_4)=\Theta(n^{3/2})$.

\begin{theorem}
\label{c4upperbound}
$\lex(n, C_4^{1243})=O(n^{3/2} \log n)$.
\end{theorem}

\begin{proof}
Let $G^L$ be an edge-ordered graph with $n$ vertices and $m$ edges avoiding $C_4^{1243}$. We assume the edges are labeled with the integers $1$ through $m$. Our goal is to bound $m$.

We call an edge-ordered subgraph of $G^L$ isomorphic to $C_4^{1234}$ an \emph{increasing 4-cycle}. Consider an increasing 4-cycle on vertices $a$, $b$, $c$ and $d$ with $L(ab)<L(bc)<L(cd)<L(da)$. We say that the \emph{width} of this increasing 4-cycle is $w(abcd)=L(da)-L(cd)+L(bc)-L(ab)$. Note that $1<w(abcd)<m$. We say that the increasing 4-cycle $abcd$ contributes the value $v=\log(m/w(abcd))$ to the pair $\{b,d\}$ of vertices and the value $-v$ to the pair $\{a,c\}$. For two distinct vertices $x,y$ of $G$, let $V(x,y)$ be the total value the pair $\{x,y\}$ received from the contributions of all increasing 4-cycles in $G$. As each 4-cycle contributes a total of zero value we clearly have
\begin{equation}\label{zero}
\sum_{x,y}V(x,y)=0,
\end{equation}
where the summation runs over all unordered pairs of distinct vertices in $G$.

We will show that $V(x,y)$ is strictly positive unless $x$ and $y$ have only a few common neighbors. This will help us bound the codegrees and the eventually the number of edges in $G$.

For a pair of distinct vertices $x$ and $y$ of $G$, let $N(x,y)$ stand for the set of \textit{common} neighbors of $x$ and $y$ in $G$. For $z\in N(x,y)$ we write $w_{xy}(z)=L(xz)-L(zy)$. Note that $w_{yx}(z)=-w_{xy}(z)$.

Any contribution to $V(x,y)$ must come from an increasing 4-cycle $xzyt$ with $z,t\in N(x,y)$. Consider first a pair of distinct vertices $z,t\in N(x,y)$ with $w_{xy}(z)$ and $w_{xy}(t)$ having the same sign. We claim that in this case $xzyt$ is an increasing 4-cycle of width  $w=w(xzyt)=|w_{xy}(z)-w_{xy}(t)|$ and it contributes $\log(m/w)$ to the pair $x,y$. By symmetry, it is enough to show this assuming that both $w_{xy}(z)$ and $w_{xy}(t)$ are positive and $L(xz)>L(xt)$. This implies $L(yz)<L(yt)$ as otherwise the 4-cycle $xzyt$ would be isomorphic to the forbidden 4-cycle $C_4^{1243}$. So we have $L(zy)<L(yt)<L(tx)<L(xz)$ making a $xzyt$ an increasing 4-cycle of width $w=L(xz)-L(tx)+L(yt)-L(zy)=w_{xy}(z)-w_{xy}(t)$ as claimed and contributing $\log(m/w)$ toward $V(x,y)$.

Now consider a pair $z,t\in N(x,y)$ with $w_{xy}(z)$ and $w_{xy}(t)$ having opposite signs, say $w_{xy}(z)<0<w_{xy}(t)$. In this case the 4-cycle $xzyt$ is not necessarily increasing, but if it is, its width is again $w=|w_{xy}(z)-w_{xy}(t)|$ and it contributes $-\log(m/w)$ toward $V(x,y)$. Indeed, the 4-cycle $xzyt$ is only increasing if either $L(yt)<L(tx)<L(xz)<L(zy)$ or $L(xz)<L(zy)<L(yt)<L(tx)$ and our assertions hold in either case.

We can calculate $V(x,y)$ by summing the above values for all distinct $z,t\in N(x,y)$:
\begin{equation}\label{value}
V(x,y)\ge\sum_{z,t}\sgn(w_{xy}(z)w_{xy}(t))\log\left(\frac m{|w_{xy}(z)-w_{xy}(t)|}\right),
\end{equation}
where the summation is for unordered pairs of distinct vertices $z,t\in N(x,y)$. We have inequality and not equality because some of the pairs may yield non-increasing 4-cycles and thus do not contribute to $V(x,y)$, but as we saw this can only happen when $\sgn(w_{xy}(z)w_{xy}(t))$ would be negative.

It will be easier to deal with $\max(|w_{xy}(z)|,|w_{xy}(t)|)$ in place of $|w_{xy}(z)-w_{xy}(t)|$ in inequality (\ref{value}). The former is the larger of the two values if the signs of $w_{xy}(z)$ and $w_{xy}(t)$ agree, but the latter is larger otherwise, so we always have

\begin{equation}\label{egyszeru}
\sgn(w_{xy}(z)w_{xy}(t))\log\left(\frac m{|w_{xy}(z)-w_{xy}(t)|}\right) >\sgn(w_{xy}(z)w_{xy}(t))\log\left(\frac m{\max(|w_{xy}(z)|,|w_{xy}(t)|)}\right).
\end{equation}

Further, the difference between the two sides of inequality (\ref{egyszeru}) is at least $1$ whenever
\begin{equation}\label{kozel}
\frac12\le\frac{w_{xy}(z)}{w_{xy}(t)}\le2.
\end{equation}
As  $1\le|w_{xy}(z)|<m$ holds for any $z\in N(x,y)$ we can partition $N(x,y)$ into $2\lceil\log m\rceil$ parts such that whenever $z$ and $t$ are from the same part, condition (\ref{kozel}) holds. This means that condition (\ref{kozel}) is satisfied for at least $d_{xy}^2/(2\lceil\log m\rceil)$ of the ordered pairs $z,t\in N(x,y)$, where $d_{xy}=|N(x,y)|$ is the codegree of $x$ and $y$ in $G$. Thus, (\ref{kozel}) is also satisfied for at least $d_{xy}^2/(4\lceil\log m\rceil)-d_{xy}$ unordered pairs of distinct vertices $z,t\in N(x,y)$. Substituting inequality (\ref{egyszeru}) in our bound (\ref{value}) and using the slack in (\ref{egyszeru}) whenever (\ref{kozel}) is satisfied, we obtain
\begin{equation}\label{value2}
V(x,y)>\sum_{z,t}\sgn(w_{xy}(z)w_{xy}(t))\log\left(\frac m{\max(|w_{xy}(z)|,|w_{xy}(t)|)}\right)+\frac{d_{xy}^2}{4\lceil\log m\rceil}-d_{xy},
\end{equation}
where the summation is for unordered pairs of distinct vertices $z,t\in N(x,y)$.

Consider now the following integral
\begin{eqnarray*}
0&\le&\int_1^m\frac1u(|\{z\in N(x,y):u>w_{xy}(z)>0\}|-|\{z\in N(x,y):-u<w_{xy}(z)<0\}|)^2\,\mathrm{d}u\\
&=&\sum_{z,t\in N(x,y)}\int_{\max(|w_{xy}(z)|,|w_{xy}(t)|)}^m\frac{\sgn(w_{xy}(z)w_{xy}(t))}u\,\mathrm{d}u\\
&=&\sum_{z,t\in N(x,y)}\sgn(w_{xy}(z)w_{xy}(t))\ln\left(\frac m{\max(|w_{xy}(z)|,|w_{xy}(t)|)}\right).
\end{eqnarray*}

The summations here are for ordered pairs $z,t$ and contains terms with $z=t$. We simply bound these latter terms by $\ln m$ and switch to binary logarithm to obtain
\begin{equation*}
\sum_{z,t \in N(x,y)}\sgn(w_{xy}(z)w_{xy}(t))\log\left(\frac m{\max(|w_{xy}(z)|,|w_{xy}(t)|)}\right)\ge-d_{xy}\log m/2,
\end{equation*}
where the summation is now for unordered pairs of distinct vertices $z,t\in N(x,y)$. With our bound (\ref{value2}) this means
\begin{equation*}
V(x,y)>\frac{d_{xy}^2}{4\lceil\log m\rceil}-d_{xy}-d_{xy}\log m/2.
\end{equation*}

It remains to sum this last bound for all unordered pairs of distinct vertices $x,y$ of $G$. On the left hand side we obtain zero by equality (\ref{zero}). With the notation $D=\sum_{x,y}d_{xy}$ we clearly have $\sum_{x,y}d_{xy}^2>2D^2/n^2$ (both summations are for unordered pairs of distinct vertices of $G$). We obtain:

\begin{equation*}
0>\frac{2D^2}{4n^2\lceil\log m\rceil}-(\log m/2+1)D,
\end{equation*}
and therefore
\begin{equation}\label{d}
D=O(n^2\log^2m).
\end{equation}

Bounding $m$ from this bound on $D$ is straightforward. Let $d_x$ denote the degree of the vertex $x$ of $G$, then we have $\sum_xd_x=2m$ and $\sum_x\binom{d_x}2=D$. By convexity we also have
$$D=\sum_x\binom{d_x}2\ge n\binom{2m/n}2=\Omega\left(\frac{m^2}n\right).$$

We obtain the bound $m=O(n^{3/2}\log n)$ claimed in the theorem by combining this last well known bound with our bound (\ref d) on $D$.
\end{proof}

\section{Concluding remarks}\label{conc}

\subsection{An application: number of unit distances among $n$ planar points in convex position}
\label{sec:application}

Tur\'an theory for edge-ordered graphs is likely to have several applications in other areas, especially in discrete geometry. As an example, we show a simple application of one of our results concerning the Tur\'an number of $P_5^{2143}$ (Theorem \ref{recur} (ii)). Erd\H{o}s and Moser asked in 1959 to determine the maximum number of point pairs among $n$ points in the plane in \emph{convex} position that can be exactly unit distance apart. If we denote this quantity by $f(n)$, then the best bounds known are $2n-7\le f(n)=O(n\log n)$, due to Edelsbrunner-Hajnal \cite{EH91}, and F\"uredi \cite{F1990}. (For a later, simpler proof of the upper bound, see \cite{BP2001}.) Here we reprove the upper bound using the theory of forbidden edge-ordered graphs. F\"uredi \cite{F1990} used forbidden submatrices, and our argument is inspired by his.

\begin{proposition}[\cite{F1990}]
The number of unit distances among $n$ points in plane in convex position is $O(n\log n)$.
\end{proposition}

\begin{proof}
Define a graph $G$ with the $n$ points in convex position in the plane as its vertices and by connecting those points that are unit distance apart.
Represent these edges as straight-line segments of length one.
Without loss of generality (rotating the plane if needed), we can assume that at least half of these line-segments have slope between $-1$ and $+1$. Keep only these edges to form a graph $G_1$, thus we have $|E(G_1)|\ge \frac{1}{2}|E(G)|$.

Now add an infinite number of vertical lines (each of infinite length) to the plane such that two  neighboring vertical lines are $\frac{7}{5}$ units apart. Keep only those edges of $G_1$ that do not cross any of these vertical lines to form the graph $G_2$. A simple probabilistic argument shows that the vertical lines can be placed in such a way that $|E(G_2)|\ge \frac{2/5}{7/5}|E(G_1)| = \frac{2}{7}|E(G_1)|\ge \frac{1}{7}|E(G)|$.

We show that any path in $G_2$ must consist of alternating steps to the left and to the right. Indeed, since any edge of $G_2$ is a unit line-segment with slope between -1 and +1, the horizontal component (or the $x$-component) of any edge of $G_2$ is at least $\frac{1}{\sqrt{2}}$, so two edges in the same direction would not fit between two vertical lines of distance $\frac{7}{5}<\sqrt{2}$.

Order the edges of $G_2$ by the slope of the respective line-segments, breaking ties arbitrarily.
\begin{claim}\label{concave}
$G_2$ is $P_5^{2143}$-free.
\end{claim}

\begin{figure}[h]
\begin{center}
\includegraphics[scale = 0.5] {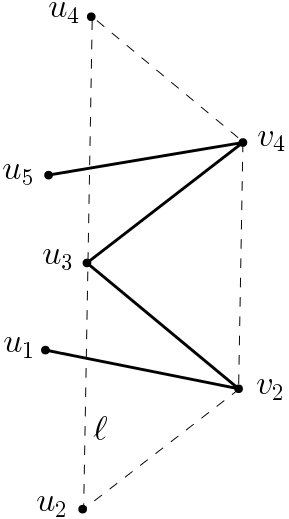}
\end{center}
\caption{Illustration for the proof of Claim \ref{concave}}
\label{convex}
\end{figure}

\begin{proof}
Suppose that we have a path $u_1v_2u_3v_4u_5$ in $G_2$ such that the slopes of its edges is in the order $v_2u_3<u_1v_2<v_4u_5<u_3v_4$.
Draw a line $\ell$ parallel to $v_2v_4$ through $u_3$, and denote the points at distance $|v_2v_4|$ from $u_3$ on $\ell$ by $u_2$ and $u_4$, respectively, such that $u_2v_2$ is a translate of $u_3v_4$ and $u_4v_4$ is a translate of $u_3v_2$.
Due to the order of the slopes, looking from $v_2$, $u_1$ must lie between $u_2$ and $u_3$, thus beyond the line $\ell$.
Similarly, $u_5$ must also lie on the other side of $\ell$ as $v_4$.
Thus $u_3$ is in the convex hull of $\{u_1,v_2,v_4,u_5\}$, contradicting our assumption that the points are in convex position.
\end{proof}
Therefore, by Theorem \ref{recur} (ii), the number of unit distances among the points in $V(G)$ is $|E(G)|\le7|E(G_2)|\le7\lex(n,P_5^{2143})=O(n\log n)$. 
\end{proof}

\begin{remark}
One may try to use that $G_2$ is also $P_5^{3142}$-free in the hope of getting a better upper bound. Currently, the best known bound is $n\log_2 n+O(n)$ due to Aggarwal \cite{A2015}. It is a very interesting question to determine whether a linear upper bound can be obtained by excluding some other edge-ordered graph(s).
\end{remark}

\subsection{Generalizations of our construction techniques}\label{gene}

In the proof of Theorem~\ref{recur} we gave recursive constructions of edge-ordered graphs $G_i$ and $G'_i$, both with $2^i$ vertices and $i2^i$ edges. $G_i$ was constructed to avoid $P_5^{1342}$, $G'_i$ was constructed to avoid $P_5^{2143}$. But the same argument shows that they avoid other edge-ordered graphs too. We can formulate the following statement.

%\tg{I formulated this as a proposition but this makes little sense without a concrete example. Can you give an example satisfying one of the conditions but not containing the corresponding 4-edge path?} 
%\dom{Well, 134567892 does not contain 1342. Any path works where the second largest edge is to the 'left' of the largest, and the smallest edges are not all to the 'right'.  In the other case we can even find a 4-edge path: 3142 is not the same as 2143. I've discussed these with Gabor today, he said they were fine.}

\begin{proposition}
Let $H$ be an edge-ordered graph on more than one vertices that has no partition of its vertex set into two non-empty parts $A$ and $B$ such that the edges between $A$ and $B$ form a matching and are larger than any other edges in the graph, while edges within $A$ are smaller than the edges within $B$. In this case the graphs $G_i$ in the proof of Theorem~\ref{recur} avoid $H$.

Similarly, the graphs $G'_i$ in the proof of Theorem~\ref{recur} avoid all edge-ordered connected graphs on more than one vertices that have no partition of their vertex sets into two non-empty sets $A$ and $B$ such that the edges between $A$ and $B$ form a matching and they are larger than any edge within $A$ and smaller than any edge within $B$.
\end{proposition}

\begin{proof} We prove the statements of the proposition by induction on $i$. They trivially hold for the single vertex graph $G_0$. Assume for a contradiction that $G_i$ avoids an edge-ordered graph $H$ but $G_{i+1}$ contains it. Recall that $G_{i+1}$ is constructed from two disjoint copies of $G_i$ by adding a perfect matching between them. So given an isomorphic copy of $H$ in $G_{i+1}$ we can partition the vertex set of $H$ according to which copy the corresponding vertex in the copy of $H$ belongs to. By the way the edge-ordering of $G_{i+1}$ was defined, this partition violates the assumption on $H$ unless one of the parts is empty. But if one of the parts is empty, then $H$ is contained in a single copy of $G_i$ violating the inductive assumption. The contradiction proves the first statement of the proposition.

For the second statement the same proof works verbatim if we replace $G_i$ and $G_{i+1}$ with $G'_i$ and $G'_{i+1}$.
\end{proof}

A similar generalized statement can be formulated about the edge-ordered graphs avoided, left-avoided and right-avoided by the edge-ordered bipartite graphs $G_i$ constructed in the proof of Theorem~\ref{thm1324}.

%\tg{No proposition this time unless you can come up with a nice example for this.}

We used a simple connection to the theory of forbidden matrix patterns in the proofs of Theorems~\ref{DSlower} and \ref{thm1432}. In general we can make an edge-ordered bipartite graph $G(M)$ from any 0-1 matrix $M$ by having a left vertex for every row, a right vertex for every column, an edge between the corresponding vertices for every $1$-entry in the matrix and ordering edges first according to their column and within a column according to the row. 

Recall that for 0-1 matrices $M$ and $P$ we say that $M$ \emph{contains} $P$ if $P$ is a submatrix of $M$ or $P$ can be obtained from a submatrix of $M$ by switching a few $1$ entries to $0$. If $M$ does not contain any pattern $P$ with $G(P)$ isomorphic to a fixed edge-ordered graph $H$, then $G(M)$ clearly avoids $H$. In the proof of Theorem~\ref{thm1432} we used the fact that there is exactly one 0-1 matrix $P$ with $G(P)$ isomorphic to $P_5^{1432}$, namely $\begin{pmatrix}0&1&1\\1&0&1\end{pmatrix}$. The same connection can be used for other patterns as well.

%\dom{it seems that lex biclique embedding of seth pettie's matrix \cite{P2011} gives a tree $H$ for which $\mathcal M_H$ has only that matrix, so we get $\lex(n,H)=\Omega(n\log n\log\log n)$.} %\gd{is this true?}

%\tg{Pettie's example is interesting as it was the first example of a vertex ordered tree with interval chromatic number 2 and extremal function not $O(n\log n)$. This is not interesting unless we want to make a challenge to find edge-ordered trees of order chromatic number 2 with high extremal function. Do we even want to make this challenge? And then we have to really check if his example translates. I am in favor of some other simpler example where this approach works. Do you have something?} 

\subsection{Open problems}
\label{sec:openprob}
In this subsection we collect some open problems. 

\begin{itemize}
 
    \item For every edge-ordered tree $T^L$ with $\chi'_<(T^L)=2$, is it true that $\lex(n, T^L)=O(n^{1+o(1)})$? This conjecture can be considered as the edge-ordered analogue of a similar conjecture on the Tur\'an number of vertex-ordered trees having interval chromatic number $2$:  In \cite{FH1992} it was conjectured that an upper bound of $O(n\log n)$ holds on the Tur\'an number of any vertex-ordered tree with interval chromatic number $2$. This conjecture has been refuted by Pettie \cite{P2011}, and a slightly weaker (and still open) version of this conjecture was given in \cite{PT2006}. 
    A partial result towards this weaker conjecture was proved in \cite{Korandi}.
    
    \item If $\chi'_<(G)$ is finite, what is its growth rate as a function of the number $n$ of vertices in $G$? It would be interesting to decide whether it grows exponentially or double exponentially in $n$.
%Is $\chi'_<(G)$, if finite, bounded by an exponential function of the number $n$ of vertices in $G$? How about $\chi_<'(\mathcal H)$ for a family of edge-ordered graphs $\mathcal H$ with each member having at most $n$ vertices? 
    In Proposition~\ref{maxorderchrom} we proved that the optimal family to consider here is the family $\mathcal K_n$ consisting of the four canonical edge-orderings of $K_n$ and gave a doubly exponential upper bound for its order chromatic number. It is not clear which single $n$-vertex edge-ordered graph has the highest finite order chromatic number. Theorem~\ref{explower} states an exponential lower bound in case of the edge-ordered graph $D_n$.

    \item Is $\lex(n, C_4^{1243})= \Theta(n^{3/2})$? Theorem \ref{c4upperbound} shows an upper bound of $O(n^{3/2} \log n)$. Is $\lex(n,P_5^{1342})=\Theta(n\log n)$?  Theorem \ref{thm1342up} shows an upper bound of $O(n\log^2 n)$. These are the only edge-orderings of $C_4$ and $P_5$ for which we do not know the order of magnitude of the Tur\'an number.

\end{itemize}

\subsection*{Acknowledgements}

We thank Tran Manh Tuan for his useful comments on the first version of this paper and for pointing us to \cite{NS}.

The research of D\'aniel Gerbner was supported by the J\'anos
    Bolyai Research Fellowship of the Hungarian Academy of Sciences and the
    National Research, Development and Innovation Office -- NKFIH under the
    grants FK 132060, K 116769, KH130371 and SNN 129364.
    
    The research of Abhishek Methuku was supported by the EPSRC, grant no. EP/S00100X/1, by IBS-R029-C1 and by the
    National Research, Development and Innovation Office - NKFIH under the
    grant K 116769. 
    
    The research of D\'aniel T.\ Nagy was supported by the
    National Research, Development and Innovation Office - NKFIH under the
    grant K 116769. 
    
    The research of D\"om\"ot\"or P\'alv\"olgyi was supported by the Lend\"ulet program of the Hungarian Academy of Sciences
    (MTA), under grant number LP2017-19/2017.
    
    The research of G\'abor Tardos was supported by the Cryptography ``Lend\"ulet''
    project of the Hungarian Academy of Sciences, by the National Research,
    Development and Innovation Office, NKFIH projects K-116769, KKP-133864 and SSN-117879, by the ERC Advanced Grant ``GeoScape'', and by the grant of Russian Government N 075-15-2019-1926. 
    
    The research of M\'at\'e Vizer was supported by the Hungarian National Research, Development and Innovation Office -- NKFIH under the grant SNN 129364, KH 130371 and KF 132060 by the J\'anos Bolyai Research Fellowship of the Hungarian Academy of Sciences and by the New National Excellence Program under the grant number \'UNKP-19-4-BME-287.

\end{document}